\documentclass[11pt,twoside,reqno]{amsart}

\usepackage{microtype}
\usepackage{cite}
\usepackage[OT1]{fontenc}
\usepackage{type1cm}
\usepackage{amssymb}
\usepackage[shortlabels]{enumitem}
\usepackage{comment}
\usepackage{xcolor}
\usepackage{dsfont}
\usepackage{mathtools}

\usepackage{geometry}
\usepackage{hyperref}
\hypersetup{
  colorlinks=true,
  linkcolor=black,
  anchorcolor=black,
  citecolor=black,
  filecolor=black,
  menucolor=red,
  runcolor=black,
  urlcolor=black,
}
\usepackage[capitalize,noabbrev,nameinlink]{cleveref}

\numberwithin{equation}{section}
\crefformat{equation}{(#2#1#3)}
\crefformat{figure}{#2Figure #1#3}
\crefformat{enumi}{(#2#1#3)}
\crefformat{theorem}{#2Theorem~#1#3}
\crefformat{lemma}{#2Lemma~#1#3}
\crefformat{proposition}{#2Proposition~#1#3}
\crefformat{corollary}{#2Corollary~#1#3}
\crefformat{claim}{#2Claim~#1#3}
\crefformat{section}{#2§#1#3}
\crefformat{subsection}{#2§#1#3}

\theoremstyle{plain}
\newtheorem{theorem}{Theorem}[section]

\newtheorem{proposition}[theorem]{Proposition}

\newtheorem{lemma}[theorem]{Lemma}

\theoremstyle{remark}

\newtheorem{example}[theorem]{Example}

\theoremstyle{definition}

\makeatletter
\newcommand{\doverline}[1]{\overline{\dbl@overline{#1}}}
\newcommand{\dbl@overline}[1]{\mathpalette\dbl@@overline{#1}}
\newcommand{\dbl@@overline}[2]{%
  \begingroup
  \sbox\z@{$\m@th#1\overline{#2}$}%
  \ht\z@=\dimexpr\ht\z@-2\dbl@adjust{#1}\relax
  \box\z@
  \ifx#1\scriptstyle\kern-\scriptspace\else
  \ifx#1\scriptscriptstyle\kern-\scriptspace\fi\fi
  \endgroup
}
\newcommand{\dunderline}[1]{\@@underline{\dbl@underline{#1}}}
\newcommand{\dbl@underline}[1]{\mathpalette\dbl@@underline{#1}}
\newcommand{\dbl@@underline}[2]{%
  \begingroup
  \sbox\z@{$\m@th#1\@@underline{#2}$}%
  \dp\z@=\dimexpr\dp\z@-2\dbl@adjust{#1}\relax
  \box\z@
  \ifx#1\scriptstyle\kern-\scriptspace\else
  \ifx#1\scriptscriptstyle\kern-\scriptspace\fi\fi
  \endgroup
}
\newcommand{\dbl@adjust}[1]{%
  \fontdimen8
  \ifx#1\displaystyle\textfont\else
  \ifx#1\textstyle\textfont\else
  \ifx#1\scriptstyle\scriptfont\else
  \scriptscriptfont\fi\fi\fi 3
}
\makeatother

\newcommand{\II}{\mathcal{I}}
\newcommand{\JJ}{\mathcal{J}}
\newcommand{\LL}{\mathcal{L}}

\newcommand{\R}{\mathbb{R}}

\newcommand{\N}{\mathbb{N}}

\newcommand{\hhh}{\mathtt{h}}
\newcommand{\iii}{\mathtt{i}}
\newcommand{\jjj}{\mathtt{j}}
\newcommand{\kkk}{\mathtt{k}}
\renewcommand{\lll}{\mathtt{l}}
\newcommand{\ppp}{\mathtt{p}}
\newcommand{\qqq}{\mathtt{q}}
\newcommand{\eps}{\varepsilon}
\newcommand{\fii}{\varphi}
\newcommand{\roo}{\varrho}
\newcommand{\ualpha}{\doverline{\alpha}}
\newcommand{\lalpha}{\dunderline{\alpha}}

\newcommand{\A}{\mathsf{A}}

\newcommand{\dd}{\,\mathrm{d}}

\renewcommand{\ge}{\geqslant}
\renewcommand{\le}{\leqslant}
\renewcommand{\geq}{\geqslant}
\renewcommand{\leq}{\leqslant}

\DeclareMathOperator{\dimm}{dim_M}
\DeclareMathOperator{\udimm}{\overline{dim}_M}

\DeclareMathOperator{\dimh}{dim_H}

\DeclareMathOperator{\dimaff}{dim_{aff}}

\DeclareMathOperator{\GL}{GL}

\DeclareMathOperator{\diam}{diam}
\DeclareMathOperator{\diag}{diag}
\DeclareMathOperator{\proj}{proj}

\DeclareMathOperator{\conv}{conv}

\renewcommand{\atop}[2]{\genfrac{}{}{0pt}{}{#1}{#2}}

\allowdisplaybreaks

\begin{document}

\title{Self-affine sponges with random contractions}

\author{Bal\'azs B\'ar\'any}
\address[Bal\'azs B\'ar\'any]
        {Department of Stochastics \\ 
        HUN-REN–BME Stochastics Research Group \\ 
        Institute of Mathematics \\  
        Budapest University of Technology and Economics \\ 
        M\H{u}egyetem rkp. 3., H-1111 Budapest, Hungary}
\email{barany.balazs@ttk.bme.hu}

\author{Antti K\"aenm\"aki}
\address[Antti K\"aenm\"aki]
        {University of Eastern Finland \\
         Department of Physics and Mathematics \\
         P.O.\ Box 111 \\
         FI-80101 Joensuu \\
         Finland}
\email{antti@kaenmaki.net}

\author{Micha\l{} Rams}
\address[Michal Rams]
        {Institute of Mathematics \\
         Polish Academy of Sciences \\
         S\'niadeckich 8 \\
         00-656 Warsaw \\
         Poland}
\email{rams@impan.pl}

\thanks{B.~B\'ar\'any was supported by the grants NKFI K142169, and the grant NKFI KKP144059 ``Fractal geometry and applications''. M.~Rams was supported by National Science Centre grant 2019/33/B/ST1/00275 (Poland). A.~K\"aenm\"aki and M.~Rams were supported by the Erd\H{o}s Center. The project was part of the semester program ``Fractals and Hyperbolic Dynamical Systems'' organized and funded by the Erd\H{o}s Center.}
\date{\today}

\begin{abstract}
  We compute the almost sure Hausdorff dimension of random self-affine sponges in $\R^d$ without imposing any separation conditions. In this context, randomness arises from the matrices in the defining semigroup, which are random yet the corresponding affine maps share a fixed point.
\end{abstract}

\maketitle

\section{Introduction and results} \label{sec:intro}

Let $\Phi = (\fii_1, \ldots, \fii_N)$ denote a tuple of contractive affine maps $\fii_i \colon \mathbb{R}^d \to \mathbb{R}^d$, defined by $\fii_i(x) = A_i x + v_i$, where $\A = (A_1, \ldots, A_N) \in \GL_d(\mathbb{R})^N$ consists of invertible $d \times d$ matrices, and $\mathsf{v} = (v_1, \ldots, v_N) \in (\mathbb{R}^d)^N$ is a collection of translation vectors. According to the seminal result of Hutchinson \cite{Hutchinson1981}, for each such $\Phi$, there exists a non-empty compact set $X \subset \mathbb{R}^d$ satisfying
\begin{equation*}
  X = \bigcup_{i=1}^N \fii_i(X).
\end{equation*}
This set $X$, known as the \emph{self-affine set}, is uniquely determined by $\Phi$. Throughout this work, we adopt the convention that any reference to a self-affine set $X$ implicitly assumes it is associated with a specific tuple $\Phi$ of affine maps that generates it. Our primary objective in this paper is to investigate the Hausdorff dimension, denoted $\dim_{\mathrm{H}}(X)$, of random variants of the self-affine set $X$.

The \emph{singular value function} of a matrix $A \in \mathrm{GL}_d(\mathbb{R})$ is defined as
\begin{equation*}
  \varphi^s(A) =
  \begin{cases}
    \alpha_1(A) \cdots \alpha_{\lfloor s \rfloor}(A) \alpha_{\lceil s \rceil}(A)^{s - \lfloor s \rfloor}, & \text{if } 0 \le s \le d, \\
    |\det(A)|^{s/d}, & \text{if } s > d,
  \end{cases}
\end{equation*}
where $0 < \alpha_d(A) \le \cdots \le \alpha_1(A) = \|A\| < 1$ represent the singular values of $A$, with $\|A\|$ denoting the operator norm. The \emph{affinity dimension} of the tuple $\A = (A_1, \ldots, A_N)$, denoted $\dimaff(\A)$, is the unique value $s$ for which the pressure function
\begin{equation*}
  P(\A,s) = \lim_{n \to \infty} \frac{1}{n} \log \sum_{i_1, \ldots, i_n} \varphi^s(A_{i_1} \cdots A_{i_n})
\end{equation*}
equals zero. According to Falconer \cite[Theorem 5.4]{Falconer1988}, the upper Minkowski dimension of the self-affine set $X$, denoted $\udimm(X)$, is bounded above by $\dimaff(\A)$. Let $\dimm(A)$ be the Minkowski dimension of a bounded set $A \subseteq \mathbb{R}^d$ and $\LL^d$ the Lebesgue measure on $\mathbb{R}^d$. Under a generic selection of translation vectors $\mathsf{v} = (v_1, \ldots, v_N) \in (\R^d)^N$, the affinity dimension serves as a lower bound for the Hausdorff dimension. Indeed, Falconer \cite[Theorem 5.3]{Falconer1988} established that if $\max_{i \in \{1, \ldots, N\}} \|A_i\| < \frac{1}{2}$, then for $\LL^{dN}$-almost every $\mathsf{v} \in (\R^d)^N$, the self-affine set $X_{\mathsf{v}}$ satisfies
\begin{equation*}
  \dimh(X_{\mathsf{v}}) = \dimm(X_{\mathsf{v}}) = \min\{d, \dimaff(\A)\}.
\end{equation*}
Initially, Falconer proved this result with a norm bound of $\frac{1}{3}$, but Solomyak \cite{Solomyak1998} later sharpened it to $\frac{1}{2}$, a threshold shown to be optimal by Edgar's example \cite{Edgar1992}. Alternatively, one can fix the translation vectors and consider a generic selection of matrices. Results in this direction, covering various settings, are detailed in B\'ar\'any, K\"aenm\"aki, and Koivusalo \cite[Theorems A and B, Proposition 2.5]{BaranyKaenmakiKoivusalo2018}.

By leveraging a separation assumption, the aforementioned results can be enhanced in the planar setting to encompass a topologically and measure-theoretically generic set of matrices. We say that a self-affine set $X$ satisfies the \emph{exponential separation condition} if there exists a constant $c > 0$ such that
\begin{equation*}
  \|\varphi_{i_1} \circ \cdots \circ \varphi_{i_n} - \varphi_{j_1} \circ \cdots \circ \varphi_{j_n}\| \ge c^n
\end{equation*}
for all $n \in \N$, whenever $i_k \ne j_k$ for some $k \in \{1, \ldots, n\}$. Additionally, $X$ satisfies the \emph{strong open set condition} if there exists an open set $U \subset \R^d$ such that $U \cap X \ne \emptyset$, $\varphi_i(U) \cap \varphi_j(U) = \emptyset$ for all $i \ne j$, and $\varphi_i(U) \subseteq U$ for all $i \in \{1, \ldots, N\}$. As shown in \cite[\S 6.2]{BaranyHochmanRapaport2019}, the strong open set condition implies exponential separation. Furthermore, we say that $X$ satisfies the \emph{fixed point condition} if the affine maps in $\Phi$ do not share a common fixed point. A tuple $\A \in \GL_d(\R)^N$ is termed \emph{proximal} if the semigroup it generates contains a matrix with a simple dominant eigenvalue, and \emph{strongly irreducible} if no finite collection $\mathcal{V}$ of proper subspaces exists such that $A_i \mathcal{V} = \mathcal{V}$ for all $i \in \{1, \ldots, N\}$. Hochman and Rapaport \cite[Theorem 1.1]{HochmanRapaport2022} proved that for a self-affine set $X \subset \R^2$ satisfying both the exponential separation condition and the fixed point condition, with an associated matrix tuple $\A$ that is proximal and strongly irreducible, it holds that
\begin{equation*}
  \dimh(X) = \dimm(X) = \min\{2, \dimaff(\A)\}.
\end{equation*}
More recently, Morris and Sert \cite[Theorem 1.5]{MorrisSert2023preprint}, building on Rapaport \cite[Theorem 1.9]{Rapaport2024}, extended this result to $\R^3$ under the strong open set condition.

Although reducible matrices represent a topologically and measure-theoretically exceptional subset, they merit attention due to their inclusion of significant examples, such as self-affine sponges. A self-affine set $X$ is termed a \emph{self-affine sponge} if its associated matrix tuple $\A = (A_1, \ldots, A_N)$ comprises diagonal matrices, with $A_i = \diag(\alpha_i^{(1)}, \ldots, \alpha_i^{(d)})$. In the planar case, these sets are commonly known as \emph{self-affine carpets}, a prominent subclass of self-affine sets that has been thoroughly explored since the 1980s; see, for instance, \cite{Bedford1984,McMullen1984}. A self-affine sponge has \emph{distinct diagonal entries} if, for every pair $j,k \in \{1,\ldots,d\}$ with $j \ne k$, there is an index $i \in \{1,\ldots,N\}$ such that $|\alpha_i^{(j)}| \ne |\alpha_i^{(k)}|$. The affinity dimension $\dimaff(\A)$ in this context is the unique $s \geq 0$ satisfying
\begin{equation} \label{eq:permutation-affinity}
  \max_{\sigma} \sum_{i=1}^N \fii_\sigma^s(A_i) = 1,
\end{equation}
where $\max_\sigma$ denotes the maximum over all permutations $\sigma$ of $\{1,\ldots,d\}$ and
\begin{equation} \label{eq:permutation-svf}
  \fii_\sigma^s(A) =
  \begin{cases}
    |\alpha^{(\sigma(1))}| \cdots |\alpha^{(\sigma(\lfloor s \rfloor))}| |\alpha^{(\sigma(\lceil s \rceil))}|^{s-\lfloor s \rfloor}, &\text{if } 0 \le s \le d, \\
    (|\alpha^{(\sigma(1))}| \cdots |\alpha^{(\sigma(d))}|)^{s/d}, &\text{if } s > d,
  \end{cases}
\end{equation}
for $A=\diag(\alpha^{(1)},\ldots,\alpha^{(d)})$. Rapaport established the following significant result in \cite[Theorem 1.3]{Rapaport2023preprint}. 

\begin{theorem} \label{thm:rapaport}
  If $X \subset \R^d$ is a self-affine sponge whose coordinate projections satisfy the exponential separation condition and has distinct diagonal entries, then
  \begin{equation*}
    \dimh(X) = \dimm(X) = \min\{d,\dimaff(\A)\}.
  \end{equation*}
\end{theorem}

Feng recently extended \cref{thm:rapaport} to self-affine measures in \cite[Theorem 1.3]{Feng2025preprint}, assuming a simple Lyapunov spectrum for the self-affine measure. Specifically, consider the set of Lyapunov exponents of the equilibrium state, defined as
\begin{equation*}
  \biggl\{ \sum_{i=1}^N \fii_\sigma^{\dimaff(\A)}(A_i) \log |\alpha_i^{(k)}| \biggr\}_{k=1}^d,
\end{equation*}
where $\sigma$ is a permutation that attains the maximum in \cref{eq:permutation-affinity}. Assuming this set consists of exactly $d$ distinct elements, Feng demonstrated the existence of a dimension-maximizing measure in \cite[Corollary 1.9]{Feng2025preprint}.

In light of \cref{thm:rapaport}, a natural question arises: can randomization yield a generic dimension result without imposing separation conditions? Jordan, Pollicott, and Simon \cite[Theorem 1.5]{JordanPollicottSimon2007} explored this by introducing a random variant of the self-affine set, perturbing the translation vector independently at each stage of its iterated construction. They proved that the Hausdorff dimension of such a randomized set almost surely equals $\dimaff(\A)$, even without requiring $\max_{i \in \{1, \ldots, N\}} \|A_i\| < \frac{1}{2}$. Jordan and Jurga \cite[Theorem 1.8]{JordanJurga} generalized this result, allowing perturbations to follow distributions with unbounded support. For random self-similar sets, Dekking, Simon, Székely, and Szekeres \cite[Theorem 4]{DekkingSimonSzekelySzekeres} proved that, on the real line, independent perturbations of translation vectors ensure that a set with positive Lebesgue measure almost surely has non-empty interior. Additionally, in \cite[Theorem 1]{DekkingSimonSzekelySzekeres}, they extended a prior result from \cite[Theorem 1]{DekkingSimonSzekely}, proving that the algebraic difference of two randomized Cantor sets on the real line almost surely contains an interval. Gu and Miao \cite{gu2024} investigated the $L^q$-dimension of randomized self-similar measures in this framework, while B\'ar\'any and Rams \cite[Theorem 1.1]{br24} examined their absolute continuity, identifying conditions for a smooth density.

Given the positive answer in the context of randomly perturbed translation vectors, the question naturally reframes itself as: what can be said when the matrices are similarly randomized while the translation vectors are held fixed? Early work in this direction includes Falconer \cite{Falconer_randomfractals}, Graf \cite{Graf}, Graf, Mauldin, and Williams \cite{GrafMauldinWilliams}, and Mauldin and Williams \cite{MauldinWilliams}. Though framed in a broader context, their key ideas are distilled in Falconer’s book \cite[Theorem~15.1]{Falconerbook}. For self-similar sets on the real line, where the contraction rates of similarities are independent, uniformly distributed random variables at each iteration, Peres, Simon, and Solomyak \cite[Theorem 2.1]{PeresSimonSolomyak} analyzed the absolute continuity of random self-similar measures, while Koivusalo \cite[Theorem 2.2]{Koivusalo} computed the almost sure Hausdorff dimension of random self-similar sets. In our main result, \cref{thm:main}, we extend this framework to higher dimensions by randomizing the matrices within the setting of \cref{thm:rapaport}. We compute the almost sure Hausdorff dimension of random self-affine sponges without requiring separation conditions. The resulting dimension formula generalizes the affinity dimension, analogous to how Koivusalo’s formula extends the similarity dimension. Before presenting this result, we introduce essential preliminaries and define the setting in detail.

\subsection{Preliminaries}
Let $N \in \N$ be such that $N \ge 2$ and write $\II = \{1,\ldots,N\}$. The set of \emph{infinite words} is $\Sigma = \II^\N$ and the \emph{left shift} $\sigma \colon \Sigma \to \Sigma$ is defined by setting $\sigma\iii = \sigma(\iii) = i_2i_3\cdots$ for all $\iii = i_1i_2\cdots \Sigma$. Let $\Sigma_*$ be the free monoid on $\{1,\ldots,N\}$. The set $\Sigma_*$ is the set of all \emph{finite words} $\Sigma_* = \{\varnothing\} \cup \bigcup_{n=1}^\infty \Sigma_n$, where $\Sigma_n = \II^n$ for all $n \in \N$ and $\varnothing$ satisfies $\varnothing \iii = \iii\varnothing = \iii$. We also set $\Sigma_0 = \{\varnothing\}$ for completeness. The concatenation of two words $\iii \in \Sigma_*$ and $\jjj \in \Sigma_* \cup \Sigma$ is denoted by $\iii\cdot\jjj$ or just $\iii\jjj$ if there is no possibility of confusion. The length of $\iii \in \Sigma_* \cup \Sigma$ is denoted by $|\iii|$. If $\iii \in \Sigma_*$, then we set $[\iii] = \{\iii\jjj \in \Sigma : \jjj\in \Sigma\}$ and call it a \emph{cylinder set}. If $\jjj \in \Sigma_* \cup \Sigma$ and $0 \le n < |\jjj|$, then we define $\jjj|_n$ to be the unique word $\iii \in \Sigma_n$ for which $\jjj \in [\iii]$. Finally, if $\iii,\jjj \in \Sigma_* \cup \Sigma$, then $\iii \land \jjj = \iii|_{|\iii\land\jjj|} = \jjj|_{|\iii\land\jjj|}$, where $|\iii\land\jjj| = \min\{n \in \N : \iii|_n \ne \jjj|_n\}-1$, is the \emph{longest common prefix} of $\iii$ and $\jjj$. We use the convention that $|$ is utilized first and then $\land$ is applied before $\cdot\,$, so for example by $\iii \land \jjj|_n \cdot \kkk$ we mean the word $(\iii \land (\jjj|_n)) \kkk$.

Recall that, for Borel probability measures $\mu$ and $\nu$ on $\R$, the \emph{convolution} of $\mu$ and $\nu$ is
\begin{equation*}
	(\mu \ast \nu)(A) = \iint \mathds{1}A(x+y) \dd\mu(x) \dd\nu(y)
\end{equation*}
for all Borel sets $A \subseteq \R$, where $\mathds{1}A$ is the characteristic function of $A$. We also set $\mu^{\ast 2} = \mu \ast \mu$ and recursively $\mu^{\ast(n+1)} = \mu^{\ast n} \ast \mu$ for all $n \in \N$. We say that the distribution $\mu$ of a real-valued random variable is \emph{eventually smooth} if there is $n \in \N$ such that the $n$-fold convolution $\mu^{\ast n}$ is absolutely continuous with continuous density, i.e.\ there exists a continuous function $g \colon \R \to \R$ such that
\begin{equation} \label{eq:smooth-def}
  \mu^{\ast n}(A) = \int_A g \dd\LL^1
\end{equation}
for all Borel sets $A \subseteq \R$.

Let $\{e_1,\ldots,e_d\}$ be the standard orthonormal basis of $\R^d$. Let $t_1,\ldots,t_N \in \R^d$ be such that the set $\{t_1 \cdot e_k,\ldots,t_N \cdot e_k\}$ is not a singleton for all $k \in \{1,\ldots,d\}$. Fix $0<\lalpha\le\ualpha<1$ and let $d \in \N$ be such that $d \ge 2$. Let $(\Omega,\mathcal{F},\mathbb{P})$ be a probability space and $\{\alpha_\iii^{(k)} : \iii\in\Sigma_* \text{ and } k\in\{1,\ldots,d\}\}$ be a collection of random variables such that
\begin{enumerate}[(R1)]
  \item\label{it:ass1} the vector-valued random variables 
  $\{(\alpha_{\iii 1}^{(\ell)},\ldots,\alpha_{\iii N}^{(\ell)}) : \iii\in\Sigma_* \text{ and } \ell\in\{1,\ldots,d\}\}$ are independent,
  \item\label{it:ass1b} the vector-valued random variables $\{(\alpha_{\iii 1}^{(\ell)},\ldots,\alpha_{\iii N}^{(\ell)}) : \iii\in\Sigma_*\}$ are identically distributed for all $\ell \in \{1,\ldots,d\}$,
  \item\label{it:ass2} $|\alpha_{\jjj}^{(k)}|\in[\lalpha,\ualpha]$ almost surely for all $k\in\{1,\ldots,d\}$ and $\jjj\in\Sigma_*$,
  \item\label{it:ass3} for every $k\in\{1,\ldots,d\}$ there is $\ell_k$ such that the distribution of $\log \alpha_{\ell_k}^{(k)}$ is eventually smooth.
\end{enumerate}
For each $k \in \{1,\ldots,d\}$ we also fix $\ell_k'$ such that $t_{\ell_k}\cdot e_k \neq t_{\ell_k'}\cdot e_k$, where $\ell_k$ is as in \ref{it:ass3}. Let
\begin{equation*}
  A_\iii =
  \begin{pmatrix}
    \alpha_\iii^{(1)} & 0 & \cdots & 0 \\
    0 & \alpha_\iii^{(2)} & \cdots & 0 \\
    \vdots & \vdots & \ddots & \vdots \\
    0 & 0 & \cdots & \alpha_\iii^{(d)}
  \end{pmatrix}
\end{equation*}
for all $\iii \in \Sigma_*$. Here $A_\varnothing = I$ is the identity matrix. Let
\begin{equation*}
  \Phi = \{ f_{\iii}(x) = A_{\iii} x + t_{i_n} (I - A_{\iii}) : \iii \in \Sigma_* \}
\end{equation*}
be a collection of random affine maps, referred to as a \emph{random iterated function system} (RIFS). There exists a compact set $B \subset \R^d$ such that $f_{\iii}(B) \subseteq B$ for every possible realization and every $\iii \in \Sigma_*$. Therefore, the random set
\begin{equation*}
  X = \bigcap_{n=0}^\infty \bigcup_{\iii \in \Sigma_n} f_{\iii|_0} \circ \cdots \circ f_{\iii|_n}(B)
\end{equation*}
is non-empty and compact for every realization. We call $X$ the \emph{self-affine sponge with random contractions}. Define the \emph{canonical projection} $\Pi \colon \Sigma \to \R^d$ by
\begin{equation*}
  \Pi(\iii) = \lim_{n \to \infty} f_{\iii|_0} \circ \cdots \circ f_{\iii|_n}(0) = \sum_{n=1}^\infty A_{\iii|_0} \cdots A_{\iii|_{n-1}} (I - A_{\iii|_n}) t_{i_n},
\end{equation*}
for all $\iii = i_1 i_2 \cdots \in \Sigma$. Consequently, the set $X$ satisfies $X = \Pi(\Sigma)$. Let $s_0$ be the unique real number satisfying
\begin{equation*}
  \max_{\sigma} \sum_{i=1}^N \mathbb{E} (\fii_\sigma^{s_0}(A_i)) = 1,
\end{equation*}
where $\fii_\sigma(A)$ is as defined in \cref{eq:permutation-svf}. By convention, whenever we refer to a self-affine sponge $X$ with random contractions, it is understood to be equipped with the definitions and notation described above.

The \emph{Hausdorff dimension} of $A \subseteq \R^d$ is defined as
\begin{align*}
  \dimh(A) = \inf\{s>0 : \;&\text{for every } \eps>0 \text{ there is } \{U_i\}_{i \in \N} \text{ such} \\
  &\text{that } A \subseteq \bigcup_{i \in \N} U_i \text{ and } \sum_{i \in \N} \diam(U_i)^s < \eps\},
\end{align*}
where $\diam(U_i)$ denotes the diameter of $U_i$. The \emph{upper Minkowski dimension} of a bounded set $A \subset \R^d$ is given by
\begin{equation} \label{eq:def-minkowski}
  \udimm(A) = \limsup_{r \downarrow 0} \frac{\log N_r(A)}{-\log r},
\end{equation}
where
\begin{equation*}
  N_r(A) = \min\{k \in \N : A \subseteq \bigcup_{i=1}^k B(x_i,r) \text{ for some } x_1,\ldots,x_k \in \R^d\}
\end{equation*}
is the smallest number of closed balls of radius $r>0$ needed to cover $A$. If the limit above exists, then the \emph{Minkowski dimension} of $A$, denoted $\dimm(A)$, is defined as this limit. It is well-known that for a compact set $A \subset \R^d$, if $\udimm(A) \le \dimh(A)$, then $\dimh(A) = \dimm(A)$.

\subsection{Main result and examples}
We are now ready to state our main result.

\begin{theorem}\label{thm:main}
  If $X \subset \R^d$ is a self-affine sponge with random contractions, then we almost surely have
  \begin{equation*}
    \dimh(X) = \dimm(X) = \min\{d,s_0\}.
  \end{equation*}
  If $s_0>d$, then $\LL^d(X)>0$ almost surely.
\end{theorem}

The theorem generalizes Koivusalo’s result \cite[Theorem 2.2]{Koivusalo} in both method and scope. At the methodological level, we follow the martingale-and-energy framework of Falconer \cite{Falconer_randomfractals} and Koivusalo \cite{Koivusalo}: construct a random measure on the symbolic space via an $L^2$-bounded martingale and combine it with a transversality estimate to obtain the lower bound. The nonconformal, diagonal setting forces two adaptations: we replace similarity ratios by the singular value function and recover multiplicativity through $n$-step subsystems, and we allow block dependence in the martingale weights. At the level of scope, the randomness is in the contractions while the fixed points are deterministic, so overlaps can be substantial. This differs from random-translation models in two ways: the translation vectors do not affect the affinity dimension, whereas the matrices do, so the target dimension depends on the randomness itself; and the fixed points may coincide, giving strong overlap without any absolute continuity of the perturbations (we only assume eventual smoothness for the contraction ratios). The lower bound uses $n$-step multiplicative subsystems to approximate the singular value function from below (giving $s_n \uparrow s_0$), a martingale measure that allows dependence within blocks, and a transversality estimate. The upper bound follows from a stopping-time covering argument, yielding $\udimm(X)\le s_0$ almost surely.

To conclude the introduction, we present a series of illustrative examples. We begin by examining a statistically self-similar set in $\R$, constructed with minimal randomness in the contraction ratios. This example, among other insights, demonstrates that our framework generalizes the result of Koivusalo~\cite{Koivusalo} already in the one-dimensional setting.

\begin{example} \label{ex:first}
	Fix $0 < \lalpha \leq \ualpha < 1$, and let $\mu$ be a Borel probability measure supported on $[\log \lalpha, \log \ualpha]$. Assume that the Fourier transform of $\mu$ exhibits polynomial decay: there exist constants $C > 0$ and $s > 0$ such that
  \begin{equation*}
    |\widehat{\mu}(\xi)| \leq C (1 + |\xi|)^{-s}
  \end{equation*}
  for all $\xi \in \R$, where
  \begin{equation*}
    \widehat{\mu}(\xi) = \int e^{i \xi x} \, \mathrm{d}\mu(x).
  \end{equation*}
  For an integer $n \geq 1$ satisfying $n s > 1$, the Fourier transform $\widehat{\mu^{*n}}$ of the $n$-fold convolution $\mu^{*n}$ lies in $L^1(\mathbb{R})$. By~\cite[Theorem~3.4]{Mattila15}, this implies that $\mu^{*n}$ is absolutely continuous with a continuous density. Polynomial Fourier decay arises naturally in many fractal measures including all self-conformal measures with a non-linearity condition; see, for instance, Algom, Rodriguez Hertz, and Wang~\cite{algom2024} and Baker and Banaji~\cite{BakerBanaji25}.
  \begin{figure}
    \centering
    \includegraphics[width=0.5\linewidth]{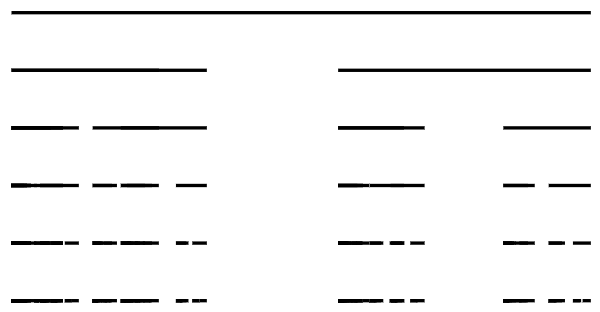}
    \caption{Illustration for the self-similar set with random contractions defined in \cref{ex:first}.}
    \label{fig:picture1}
  \end{figure}

  Let $N \in \N$ be with $N \geq2 $, and write $\{1, \ldots, N\}^* = \{\varnothing\} \cup \bigcup_{n=1}^\infty \{1, \ldots, N\}^n$. Consider a collection of independent and identically distributed random variables $\{ \alpha_{\iii 1} : \iii \in \{1, \ldots, N\}^* \}$, where the distribution of $\log \alpha_{\iii 1}$ is given by the measure $\mu$. For each $i \in \{2,\ldots,N\} $, fix deterministic constants $0 < \alpha_i < 1$, and set $\alpha_{\iii i} = \alpha_i$ for all $\iii \in \{1, \ldots, N\}^*$. Additionally, let $t_1,\ldots,t_N \in \R$ be the fixed points such that $t_i \ne t_j$ for some $i \ne j$. The resulting self-similar set $X$ with random contractions clearly satisfies assumptions~\ref{it:ass1}--\ref{it:ass3}. Notably, each iteration involves only one random cylinder, distinguishing this framework from previously studied settings and extending the scope of existing results.

  Consider an example where the random variables $\{\alpha_{\iii 1} : \iii \in \{1,2,3\}^*\}$ are independently and uniformly distributed on $[\frac{1}{3}, \frac{1}{2}]$. Define the similarities by setting
  \begin{equation*}
    f_{\iii 1}(x) = \alpha_{\iii 1} x + (1 - \alpha_{\iii 1}), \quad f_{\iii 2}(x) = \frac{x}{3}, \quad f_{\iii 3}(x) = \frac{x}{4}
  \end{equation*}
  for all $\iii \in \{1,2,3\}^*$ and $x \in \R$. Despite exact overlaps in the system, such as $f_{\iii 2} \circ f_{\iii 3} = f_{\iii 3} \circ f_{\iii 2}$ for all $\iii \in \{1,2,3\}^*$, the Hausdorff dimension of the self-similar set $X$ with random contractions is almost surely $s = 1$, the unique solution to
  \begin{equation*}
    \int_{\frac{1}{3}}^{\frac{1}{2}} 6 x^s \dd x + \frac{1}{3^s} + \frac{1}{4^s} = 1.
  \end{equation*}
  The first five iterates of the random iterated function system $\Phi = \{ f_{\iii 1}, f_{\iii 2}, f_{\iii 3} : \iii \in \{1,2,3\}^* \}$ are illustrated in~\cref{fig:picture1}.
\end{example}

Our second example draws inspiration from two sources: the higher-dimensional generalization of the random cookie cutter set discussed in Falconer’s book \cite[\S 15]{Falconerbook}; and a probabilistic extension of the generalized $4$-corner set studied by B\'ar\'any \cite{Barany12}.

\begin{example}\label{ex:4corner}
  Consider the unit square $[0,1]^2$. Place a random rectangle at each corner of the square such that the side lengths are absolutely continuous with continuous densities and the rectangles are pairwise disjoint. Repeat this process independently within each rectangle ad infinitum and consider the intersection of the unions of rectangles at each step. More precisely, consider the following random system: Let $\delta>0$ and for each $\iii\in\{1,2,3,4\}^*$ and $k \in \{1,2\}$, let $(\alpha_{\iii1}^{(k)},\alpha_{\iii2}^{(k)},\alpha_{\iii3}^{(k)},\alpha_{\iii4}^{(k)})$ be independent and identically distributed, absolutely continuous vector-valued random variables with continuous joint density such that $\mathbb{P}(\alpha_{\iii}^{(k)}\geq\delta)=1$ and
	\begin{equation}\label{eq:probcond}
	\begin{split}
    \mathbb{P}(\max\{\alpha_{\iii 1}^{(1)},\alpha_{\iii 2}^{(1)}\}+\max\{\alpha_{\iii 3}^{(1)},\alpha_{\iii 4}^{(1)}\}\leq 1)&=1,\\
    \mathbb{P}(\max\{\alpha_{\iii 1}^{(2)},\alpha_{\iii 4}^{(2)}\}+\max\{\alpha_{\iii 2}^{(2)},\alpha_{\iii 3}^{(2)}\}\leq 1)&=1.
	\end{split}
	\end{equation}
	Define the affine maps by setting
  \begin{equation} \label{eq:rifs2}
  \begin{split}
    f_{\iii1}(x,y) &= (\alpha_{\iii1}^{(1)}x,\alpha_{\iii1}^{(2)}y), \\
    f_{\iii2}(x,y) &= (\alpha_{\iii2}^{(1)}x,\alpha_{\iii2}^{(2)}y+(1-\alpha_{\iii2}^{(2)})), \\
    f_{\iii3}(x,y) &= (\alpha_{\iii3}^{(1)}x+(1-\alpha_{\iii3}^{(1)}),\alpha_{\iii3}^{(2)}y+(1-\alpha_{\iii3}^{(2)})), \\ 
    f_{\iii4}(x,y) &= (\alpha_{\iii4}^{(1)}x+(1-\alpha_{\iii4}^{(1)}),\alpha_{\iii4}^{(2)}y),
  \end{split}
  \end{equation}
  for all $\iii \in \{1,2,3,4\}^*$ and $(x,y) \in \R^2$. The self-affine sponge with random contractions is given by
  \begin{equation*}
    X=\bigcap_{n=0}^\infty\bigcup_{\iii\in\{1,2,3,4\}^n}f_{\iii|_0}\circ\cdots\circ f_{\iii|_n}([0,1]^2).
  \end{equation*}
  The system satisfies assumptions~\ref{it:ass1}--\ref{it:ass3} without requiring the condition~\eqref{eq:probcond}. The condition~\eqref{eq:probcond} ensures that the rectangles at each level are pairwise disjoint, making the system particularly relevant for constructions like the random cookie cutter set or the $4$-corner set, though it is not essential for the dimension result.
  \begin{figure}
    \centering
    \includegraphics[width=0.3\linewidth]{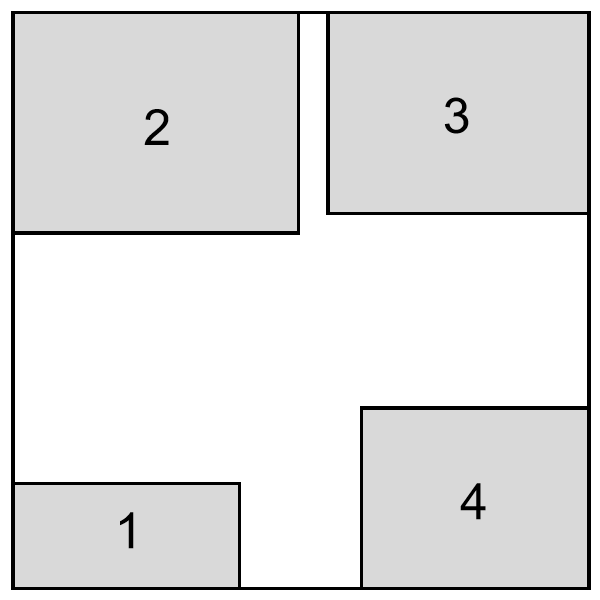}\quad\includegraphics[width=0.3\linewidth]{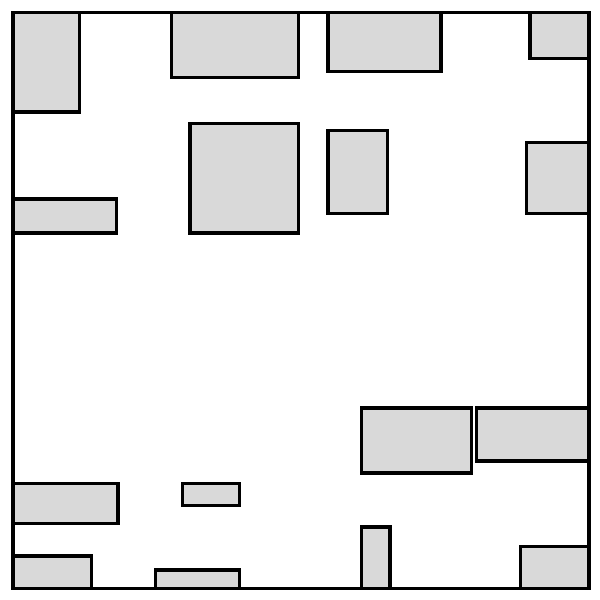}\quad\includegraphics[width=0.3\linewidth]{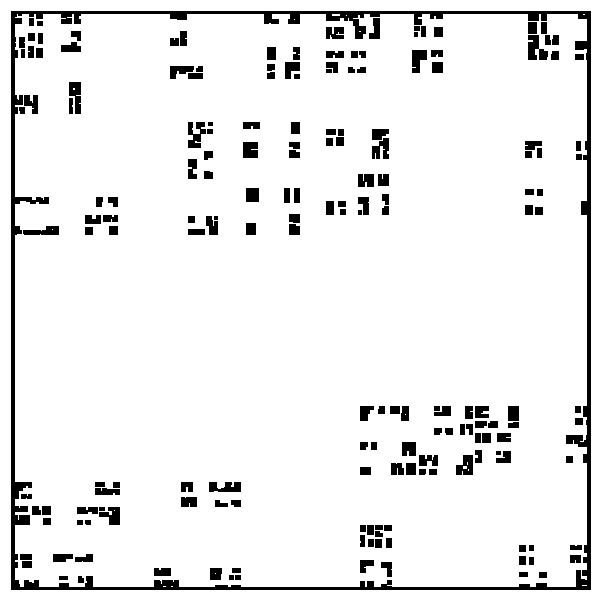}
    \caption{The first and second iterates of the RIFS defined in \eqref{eq:rifs2}, and the random 4-corner set $X$ of \cref{ex:4corner}.}
    \label{fig:picture2}
  \end{figure}

  For example, one can choose $\alpha_{\iii}^{(k)}$ independent and identically distributed with distribution uniform on $[\frac{1}{10},\frac12]$. In this case, the Hausdorff dimension of $X$ is almost surely $s \approx 1.14273$, which is the unique solution of the equation
  \begin{equation*}
    4\biggl(\int_{\frac{1}{10}}^{\frac12}\frac52x\dd x\biggr)\biggl(\int_{\frac{1}{10}}^{\frac12}\frac52x^{s-1}\dd x\biggr)=1.
  \end{equation*}
  The first two iterates of the random iterated function system $\Phi = \{f_{\iii 1}, f_{\iii 2}, f_{\iii 3}, f_{\iii 4} : \iii \in \{1,2,3,4\}^*\}$ are illustrated in~\cref{fig:picture2}.
\end{example}

For the last example, we modify \cref{ex:4corner} to introduce a more pronounced alignment structure in the random iterated function system.

\begin{example}\label{ex:mod4corner}
  \begin{figure}
    \centering
    \includegraphics[width=0.3\linewidth]{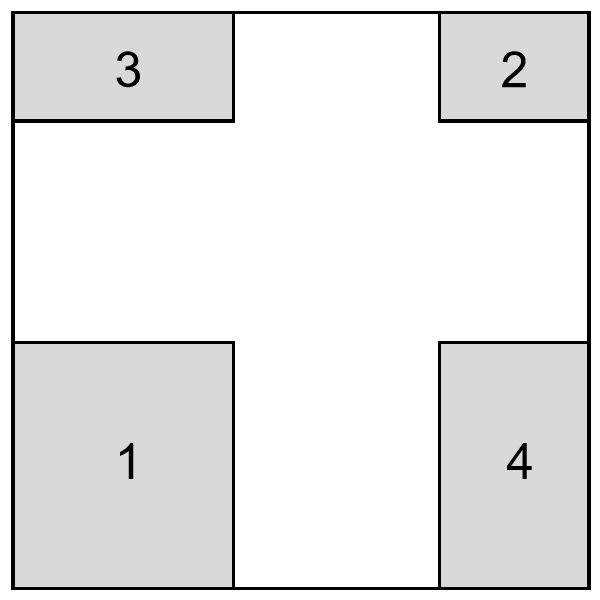}\quad\includegraphics[width=0.3\linewidth]{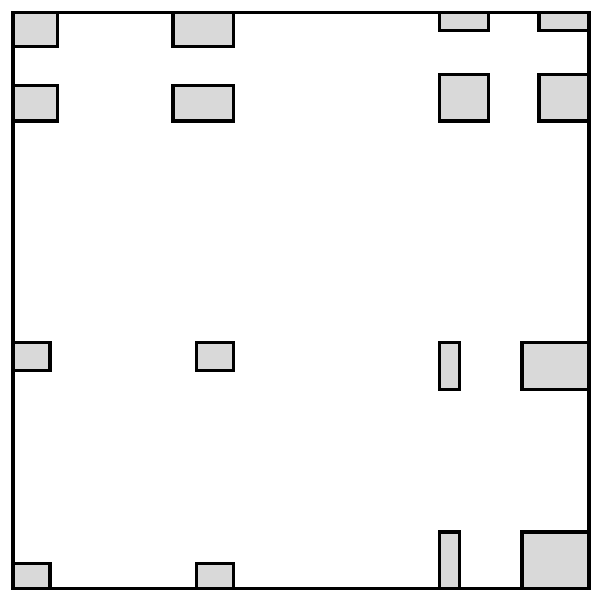}\quad\includegraphics[width=0.3\linewidth]{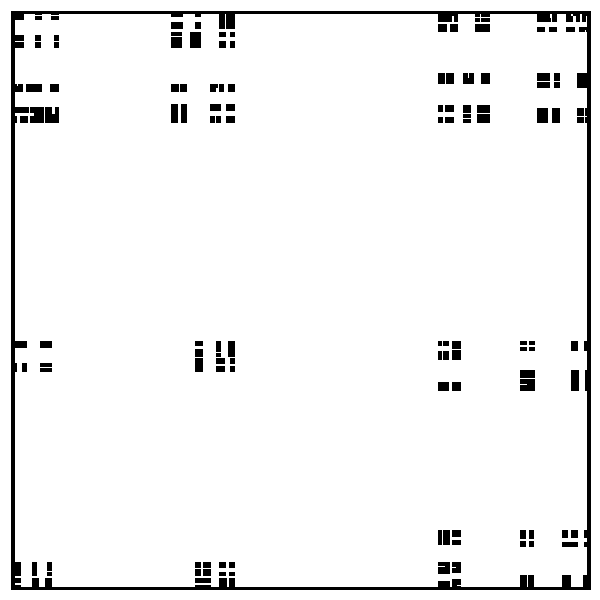}
    \caption{The first and second iterates of the RIFS defined in \eqref{eq:rifs2b}, and the random 4-corner set $X$ of \cref{ex:mod4corner}.}
    \label{fig:picture3}
  \end{figure}

  Let $\delta>0$ and for each $\iii\in\{1,2,3,4\}^*$, let $\{(\alpha_{\iii}^{(1)},\alpha_{\iii}^{(2)}):\iii\in\{1,2,3,4\}^*\}\cup\{(\alpha_{\iii}^{(3)},\alpha_{\iii}^{(4)}):\iii\in\{1,2,3,4\}^*\}$ be independent, absolutely continuous vector-valued random variables such that $\{(\alpha_{\iii}^{(1)},\alpha_{\iii}^{(2)}):\iii\in\{1,2,3,4\}^*\}$ are identically distributed with continuous joint density, $\{(\alpha_{\iii}^{(3)},\alpha_{\iii}^{(4)}):\iii\in\{1,2,3,4\}^*\}$ are identically distributed with continuous joint density, $\mathbb{P}(\alpha_{\iii}^{(k)}\geq\delta)=1$, and
  \begin{equation*}
    \mathbb{P}(\alpha_{\iii}^{(1)}+\alpha_{\iii}^{(2)}\leq 1)=1=\mathbb{P}(\alpha_{\iii}^{(3)}+\alpha_{\iii}^{(4)}\leq 1).
  \end{equation*}
  Define the affine maps by setting
  \begin{equation} \label{eq:rifs2b}
  \begin{split}
    f_{\iii1}(x,y) &= (\alpha_{\iii}^{(1)}x,\alpha_{\iii}^{(3)}y), \\
    f_{\iii2}(x,y) &= (\alpha_{\iii}^{(2)}x+(1-\alpha_{\iii}^{(2)}),\alpha_{\iii4}^{(4)}y+(1-\alpha_{\iii}^{(4)})), \\
    f_{\iii3}(x,y) &= (\alpha_{\iii}^{(1)}x,\alpha_{\iii}^{(4)}y+(1-\alpha_{\iii}^{(4)})), \\ 
    f_{\iii4}(x,y) &= (\alpha_{\iii}^{(2)}x+(1-\alpha_{\iii}^{(2)}),\alpha_{\iii}^{(3)}y),
  \end{split}
  \end{equation}
  for all $\iii \in \{1,2,3,4\}^*$ and $(x,y) \in \R^2$. The system again satisfies the assumptions \ref{it:ass1}--\ref{it:ass3}. By taking $\alpha_{\iii}^{(k)}$ independent and identically distributed with distribution uniform on $[\frac{1}{10},\frac12]$, we see that the almost sure dimension value is the same as in \cref{ex:4corner}. The first two iterates of the random iterated function system $\Phi = \{f_{\iii 1}, f_{\iii 2}, f_{\iii 3}, f_{\iii 4} : \iii \in \{1,2,3,4\}^*\}$ are illustrated in~\cref{fig:picture3}.
\end{example}

\subsection{Further discussions}
It is natural to ask whether Theorem~\ref{thm:main} extends to random matrices beyond the diagonal case. In such a generality, the critical parameter would likely be defined by a subadditive pressure of an expected singular value function rather than the explicit formula for $s_0$. Two structural inputs of our proof are tied to diagonality. First, for diagonal matrices the functions $\fii_\sigma^s$ are multiplicative, which makes the $n$-step subsystem construction and the martingale measure possible; for general matrices the singular value function is only submultiplicative, and a submartingale limit may degenerate to the zero measure. Second, our transversality estimate relies on independence of the coordinate projections; in the non-diagonal case no coordinate system is preserved and a product-type estimate is unavailable.

Even in the planar case with distinct real eigenvalues, new difficulties appear. Each matrix is diagonalizable, but the eigenbasis varies with the word, so there is no common invariant splitting. Overlaps can align along random stable directions, and the coordinatewise reduction to one-dimensional transversality breaks down. Random matrix theory (of Oseledets and Furstenberg) provides Lyapunov exponents and random invariant directions for products in $\GL_2(\R)$, and one may hope to use Ledrappier-Young type formulas for measures. However, these tools do not yield the uniform transversality required for the lower bound, nor do they provide a nondegenerate random measure compatible with the submultiplicativity of the singular value function. In this sense, the planar diagonalizable case already seems to require genuinely new ideas and is indicative of the difficulty of the general non-diagonal setting.

We expect the methods to be more flexible in intermediate settings such as block-diagonal or triangular random matrices with dominated splitting, where some multiplicativity survives within blocks and partial transversality might still be available. In forthcoming work of the first author (with Ryan Bushling) on random fractal interpolation functions with lower-triangular matrices, one already sees that establishing a transversality lemma becomes technically delicate when the weakly contracting direction is not fixed.

\section{Construction of random measures} \label{sec:measure}

In this section, we construct a random measure on the symbolic space that matches the distribution of the singular function $\fii^{s_0}$. This measure is used in \cref{sec:lower-bound} to establish an almost sure lower bound for the Hausdorff dimension. The construction employs a martingale argument, inspired by Falconer \cite{Falconer_randomfractals} and Koivusalo \cite{Koivusalo}. The presentation is designed to be self-contained and comprehensive. 

Let $(\Omega,\mathcal{F},\mathbb{P})$ be a probability space and let $\{p_\iii : \iii \in \Sigma_*\}$ be a collection of positive random variables such that
\begin{enumerate}[(P1)]
	\item\label{it:dist} the vector valued random variables $\{(p_{\iii 1},\ldots,p_{\iii N}):\iii\in\Sigma_*\}$ are independent and identically distributed,
	\item\label{it:expect} $\mathbb{E}(\sum_{j=1}^Np_j)=1$ and $\mathbb{E}(\sum_{j=1}^Np_j^2)<1$.
\end{enumerate}
Let $\mathcal{F}_k = \{p_\iii^{-1}(A) : \iii \in \bigcup_{n=0}^k \Sigma_n \text{ and } A \subseteq \R \text{ is a Borel set}\}$ be the $\sigma$-algebra generated by $\{p_{\iii} : \iii \in \bigcup_{n=0}^k \Sigma_n\}$ and write
\begin{equation*}
  \overline{p}_{\iii}^\jjj=p_{\jjj\cdot\iii|_1}\cdots p_{\jjj\iii}
\end{equation*}
for all $\iii,\jjj\in\Sigma_*$.

\begin{theorem}\label{thm:measure}
  If $\{p_\iii : \iii \in \Sigma_*\}$ is a collection of positive random variables satisfying \ref{it:dist}--\ref{it:expect}, then there exists a random finite Borel measure $\mu$ on $\Sigma$ such that
	\begin{enumerate}
		\item $0<\mu(\Sigma)<\infty$ almost surely,
		\item $\mathbb{E}(\mu([\iii])\,|\,\mathcal{F}_k)=\overline{p}_{\iii}^{\varnothing}$ for every $\iii\in\Sigma_k$,
		\item $\mathbb{E}(\sum_{\iii \in \Sigma_k}\mu([\iii]))=1$,
	\end{enumerate}
  for all $k \in \N$.
\end{theorem}

\begin{proof}
	Write
	$$
    X_k^{\iii} = \sum_{\jjj \in \Sigma_k}\overline{p}_{\jjj}^{\iii}
	$$
  for all $\iii\in\Sigma_*$ and $k \in \N$. We will show that for every $\iii\in\Sigma_*$, $X_k^{\iii}$ forms an $L^2$-bounded martingale with respect to $\mathcal{F}_{k+|\iii|}$ and hence, almost surely for every $\iii\in\Sigma_*$, there exist $L^2$-random variables $X^{\iii}$ such that
	\begin{enumerate}
		\item $X_k^{\iii}\to X^{\iii}$ almost surely and in $L^2$,
		\item\label{it:formeas} $X^{\iii}=\sum_{\jjj \in \Sigma_n}\overline{p}_{\jjj}^{\iii}X^{\iii\jjj}$,
		\item $\mathbb{E}(X^{\iii})=1$,
		\item $X^{\iii}>0$ almost surely.
	\end{enumerate}
By defining $\mu$ on cylinder sets via $\mu([\iii])=\overline{p}_{\iii}^\varnothing X^{\iii}$, the Hahn-Kolmogorov theorem (see~\cite[Theorem~1.7.8]{Tao}) ensures, in conjunction with~\eqref{it:formeas}, that $\mu$ extends uniquely to a well-defined finite Borel measure and has the desired properties.

First, let us show that $X^{\iii}_k$ forms a martingale. This follows since
\begin{align*}
	\mathbb{E}(X_{k+1}^{\iii}\,|\,\mathcal{F}_{|\iii|+k})&=\mathbb{E}\biggl(\sum_{\jjj \in \Sigma_{k+1}}\overline{p}^{\iii}_{\jjj}\,\Big|\,\mathcal{F}_{|\iii|+k}\biggr)=\mathbb{E}\biggl(\sum_{\jjj \in \Sigma_k}\sum_{j=1}^N\overline{p}^{\iii}_{\jjj}p_{\iii\jjj j}\,\Big|\,\mathcal{F}_{|\iii|+k}\biggr)\\
	&=\sum_{\jjj \in \Sigma}\overline{p}^{\iii}_{\jjj}\mathbb{E}\biggl(\sum_{j=1}^Np_{\iii\jjj j}\,\Big|\,\mathcal{F}_{|\iii|+k}\biggr)=\sum_{\jjj \in \Sigma_k}\overline{p}^{\iii}_{\jjj}\mathbb{E}\biggl(\sum_{j=1}^Np_{\iii\jjj j}\biggr)\\
	&=\sum_{\jjj \in \Sigma_k}\overline{p}^{\iii}_{\jjj}=X^{\iii}_k,
\end{align*}
where we used both \ref{it:dist} and \ref{it:expect}. Let us next show that $X_k^{\iii}$ is $L^2$-bounded. Observe that
\begin{align}
  \mathbb{E}((X_{k+1}^\iii)^2)&=\mathbb{E}\biggl(\sum_{\jjj \in \Sigma_{k+1}}\sum_{\kkk \in \Sigma_{k+1}}\overline{p}^{\iii}_{\jjj}\overline{p}^{\iii}_{\kkk}\biggr)
  =\mathbb{E}\biggl(\sum_{\jjj \in \Sigma_k}\sum_{\kkk \in \Sigma_k}\sum_{j=1}^N\sum_{k=1}^N\overline{p}^{\iii}_{\jjj}\overline{p}^{\iii}_{\kkk}p_{\iii\jjj j}p_{\iii\kkk k}\biggr)\notag\\
  &=\sum_{\jjj \in \Sigma_k}\sum_{\kkk \in \Sigma_k}\mathbb{E}(\overline{p}^{\iii}_{\jjj}\overline{p}^{\iii}_{\kkk})\mathbb{E}\biggl(\sum_{j=1}^N\sum_{k=1}^Np_{\iii\jjj j}p_{\iii\kkk k}\biggr) \label{eq:measure1}\\
  &=\sum_{\jjj \in \Sigma_k}\mathbb{E}((\overline{p}^{\iii}_{\jjj})^2)\mathbb{E}\biggl(\sum_{j=1}^N\sum_{k=1}^Np_{\iii\jjj j}p_{\iii\jjj k}\biggr)
  +\sum_{\atop{\jjj,\kkk\in\Sigma_k}{\jjj\neq\kkk}}\mathbb{E}(\overline{p}^{\iii}_{\jjj}\overline{p}^{\iii}_{\kkk})\mathbb{E}\biggl(\sum_{j=1}^N\sum_{k=1}^Np_{\iii\jjj j}p_{\iii\kkk k}\biggr). \notag
\end{align}
By \ref{it:dist} and \ref{it:expect}, we have
\begin{equation} \label{eq:measure2}
\begin{split}
  \sum_{\atop{\jjj,\kkk\in\Sigma_k}{\jjj\neq\kkk}}\mathbb{E}(\overline{p}^{\iii}_{\jjj}\overline{p}^{\iii}_{\kkk})\mathbb{E}\biggl(\sum_{j=1}^N\sum_{k=1}^Np_{\iii\jjj j}p_{\iii\kkk k}\biggr)&=\sum_{\atop{\jjj,\kkk\in\Sigma_k}{\jjj\neq\kkk}}\mathbb{E}(\overline{p}^{\iii}_{\jjj}\overline{p}^{\iii}_{\kkk})\sum_{j=1}^N\sum_{k=1}^N\mathbb{E}(p_{\iii\jjj j})\mathbb{E}(p_{\iii\kkk k})\\
  &=\sum_{\atop{\jjj,\kkk\in\Sigma_k}{\jjj\neq\kkk}}\mathbb{E}(\overline{p}^{\iii}_{\jjj}\overline{p}^{\iii}_{\kkk})\mathbb{E}\biggl(\sum_{j=1}^Np_{\iii\jjj j}\biggr)\mathbb{E}\biggl(\sum_{k=1}^Np_{\iii\kkk k}\biggr)\\
  &=\sum_{\atop{\jjj,\kkk\in\Sigma_k}{\jjj\neq\kkk}}\mathbb{E}(\overline{p}^{\iii}_{\jjj}\overline{p}^{\iii}_{\kkk})=\mathbb{E}((X_k^{\iii})^2)-\sum_{\jjj \in \Sigma_k}\mathbb{E}((\overline{p}^{\iii}_{\jjj})^2),
\end{split}
\end{equation}
where in the last equality we used \ref{it:dist}, \ref{it:expect}, and the definition of $X_k^{\iii}$. Furthermore,
\begin{equation} \label{eq:measure3}
\begin{split}
  \sum_{\jjj \in \Sigma_k}\mathbb{E}((\overline{p}^{\iii}_{\jjj})^2)&=\sum_{\jjj \in \Sigma_k}\mathbb{E}(p_{\iii\cdot\jjj|_1}^2 \cdots p_{\iii\jjj}^2)=\sum_{\jjj \in \Sigma_k}\mathbb{E}(p_{\iii\cdot\jjj|1}^2) \cdots \mathbb{E}(p_{\iii\jjj}^2)\\
  &=\sum_{\jjj \in \Sigma_k}\mathbb{E}(p_{j_1}^2) \cdots \mathbb{E}(p_{j_k}^2)=\biggl(\mathbb{E}\biggl(\sum_{j=1}^Np_j^2\biggr)\biggr)^k,
\end{split}
\end{equation}
where we used \ref{it:dist}. Thus, putting \cref{eq:measure1} and \cref{eq:measure2} together yields
\begin{align*}
	\mathbb{E}((X_{k+1}^\iii)^2)&=\sum_{\jjj \in \Sigma_k}\mathbb{E}((\overline{p}^{\iii}_{\jjj})^2)\mathbb{E}\biggl(\sum_{j=1}^N\sum_{k=1}^Np_{\iii\jjj j}p_{\iii\jjj k}-1\biggr)+\mathbb{E}((X_{k}^\iii)^2)\\
	&=\biggl(\mathbb{E}\biggl(\sum_{j=1}^Np_j^2\biggr)\biggr)^k\mathbb{E}\biggl(\sum_{j=1}^N\sum_{k=1}^Np_{j}p_{k}-1\biggr)+\mathbb{E}((X_{k}^\iii)^2),
\end{align*}
where in the last equality we used \cref{eq:measure3} and \ref{it:dist}. Then, by induction, we have
\begin{equation*}
  \mathbb{E}((X_{k}^\iii)^2)\leq\frac{\mathbb{E}\bigl(\sum_{j=1}^N\sum_{k=1}^Np_{j}p_{k}-1\bigr)}{1-\mathbb{E}\bigl(\sum_{j=1}^Np_j^2\bigr)}<\infty
\end{equation*}
for all $k\in\N$, which had to be shown. By Doob's $L^2$-martingale convergence theorem \cite[Theorem~VII.4.1]{Doob}, we now have that $\mathbb{E}(X^\iii)=\mathbb{E}(X^\iii_k)=1$ for all $k \in \N$. Furthermore, since $X_{k+1}^{\iii}=\sum_{j=1}^Np_{\iii j}X^{\iii j}_{k}$ by definition, we also have $X^{\iii}=\sum_{j=1}^Np_{\iii j}X^{\iii j}$ almost surely for every $\iii\in\Sigma_*$.

To see that $X^\iii>0$ almost surely, observe that $X^{\iii}$ and $X^{\jjj}$ have the same distribution for every $\iii,\jjj\in\Sigma_*$. Indeed, one can see by induction that the distribution of $X_k^\iii$ and $X_k^{\jjj}$ are identical for all $k\in\N$. On the other hand, for every $|\iii|=|\jjj|$ with $\iii\neq\jjj$ and $k\in\N$, $X_k^{\iii}$ and $X_k^{\jjj}$ are independent and so are $X^{\iii}$ and $X^{\jjj}$. Let $q=\mathbb{P}(X^\iii=0)$. Since $\mathbb{E}(X^{\iii})=1$, we have $0\le q<1$. Hence,
\begin{align*}
  q&=\mathbb{P}(X^{\iii}=0)=\mathbb{P}\biggl(\sum_{j=1}^Np_{\iii j}X^{\iii j}=0\biggr)\\
  &=\mathbb{P}\biggl(\bigcap_{j=1}^N\{X^{\iii j}=0\}\biggr)=\prod_{j=1}^N\mathbb{P}(X^{\iii j}=0)=q^N
\end{align*}
showing that $q=0$ and finishing the proof.
\end{proof}

\section{Transversality over subsystems} \label{sec:trans}

In this section, we establish a transversality property, following Jordan, Pollicott, and Simon \cite{JordanPollicottSimon2007}, to derive a lower bound for the Hausdorff dimension. We start with the simpler one-dimensional case, deferring the general setting to \cref{sec:space}.

\subsection{Transversality on the line} \label{sec:line}
Let $t_1,\ldots,t_N \in \R$ be such that $t_1 \le \cdots \le t_N$ with $t_1 < t_N$. Fix $0<\lalpha\le \ualpha<1$ and let $\{\alpha_\iii : \iii \in \Sigma_*\}$ be a collection of random variables in the probability space $(\Omega,\mathcal{F},\mathbb{P})$ such that
\begin{enumerate}[(R1')]
  \item\label{it:ass1R} the vector valued random variables $\{(\alpha_{\iii 1},\ldots,\alpha_{\iii N}) : \iii \in \Sigma_* \}$ are independent and identically distributed,
  \item\label{it:ass2R} $|\alpha_{\iii}^{(k)}|\in[\lalpha,\ualpha]$ almost surely for all $k\in\{1,\ldots,d\}$ and $\iii\in\Sigma_*$,
  \item\label{it:ass3R} there is $\ell$ such that the distribution of $\log \alpha_{\ell}$ is eventually smooth.
\end{enumerate}
We also fix $\ell'$ such that $t_\ell \neq t_{\ell'}$, where $\ell$ is as in \ref{it:ass3R}. At each $\iii\in\Sigma_*$, we consider a randomly chosen homothety $f_\iii$ having fixed point $t_{i_{|\iii|}}$ such that
\begin{equation}\label{eq:ssdef}
  f_{\iii}(x) = \alpha_\iii x + (1-\alpha_\iii)t_{i_{|\iii|}}.
\end{equation}
To simplify notation, we define $\overline{\alpha}_{\iii}^{\jjj} = \alpha_{\jjj\cdot\iii|_1} \cdots \alpha_{\jjj\cdot\iii|_{|\iii|}}$ and $\overline{f}_\iii^\jjj=f_{\jjj\cdot\iii|_1}\circ\cdots\circ f_{\jjj\cdot\iii|_{|\iii|}}$ for all $\iii,\jjj\in \Sigma_*$. We also write $\overline{f}_\iii = \overline{f}_\iii^\varnothing$, $\overline{\alpha}_\iii = \overline{\alpha}_\iii^\varnothing$, and $\alpha_\varnothing=\overline{\alpha}^\jjj_\varnothing\equiv1$. It is clear that, if $\jjj \in \Sigma_*$ and $\iii \in \Sigma$, then we have
\begin{align*}
  |\overline{f}_{\iii|_n}^\jjj(x)-\overline{f}_{\iii|_n}^\jjj(y)| &= |f_{\jjj\cdot\iii|_{1}} \circ \cdots \circ f_{\jjj\cdot\iii|_{n}}(x)-f_{\jjj\cdot\iii|_{1}} \circ \cdots \circ f_{\jjj\cdot\iii|_{n}}(y)| \\
  &= |\overline{\alpha}^\jjj_{\iii|_n}||x-y|\leq \doverline{\alpha}^n|x-y|
\end{align*}
for all $x,y\in\R$ and $n \in \N$. Hence, for every $\jjj\in\Sigma_*$, we can define the \emph{canonical projection} $\pi_{\jjj} \colon \Sigma \to \R$ by setting
\begin{equation}\label{eq:defnatproj11}
\begin{split}
  \pi_{\jjj}(\iii) &=\lim_{n\to\infty}\overline{f}_{\iii|_n}^\jjj(0)= \lim_{n \to \infty} f_{\jjj\cdot\iii|_{1}} \circ \cdots \circ f_{\jjj\cdot\iii|_{n}}(0)\\
  &= \sum_{n=1}^\infty \alpha_{\jjj\cdot\iii|_1} \cdots \alpha_{\jjj\cdot\iii|_{n-1}} (1-\alpha_{\jjj\cdot\iii|_n})t_{i_n} = \sum_{n=1}^\infty \overline{\alpha}_{\iii|_{n-1}}^{\jjj} (1-\alpha_{\jjj\cdot\iii|_n})t_{i_n}
\end{split}
\end{equation}
for all $\iii \in \Sigma$. Writing $\pi = \pi_{\varnothing}$, we observe that $\pi(\Sigma)$ is a self-affine sponge with random contractions in the real line. In this case, we also call $\pi(\Sigma)$ a \emph{self-similar set with random contractions}.

\begin{lemma}\label{lem:inint}
  If $X \subset \R$ is a self-similar set with random contractions, then, independently of the realization of the random variables,
  \begin{equation*}
    \pi_{\jjj}(\iii) \in \Bigl[t_1-\frac{\ualpha(t_N-t_1)}{1-\ualpha},t_N+\frac{\ualpha(t_N-t_1)}{1-\ualpha}\Bigr]
  \end{equation*}
  for all $\jjj\in\Sigma_*$ and $\iii\in\Sigma$.
\end{lemma}

\begin{proof}
  Since
	\begin{align*}
		\overline{f}_{\iii|_n}^\jjj(x) &= f_{\jjj\cdot\iii|_{1}} \circ \cdots \circ f_{\jjj\cdot\iii|_{n}}(x) \\
		&= \sum_{k=1}^{n-1} \overline{\alpha}_{\iii|_{k-1}}^\jjj(1-\alpha_{\jjj\cdot\iii|_{k}})t_{i_k} + \overline{\alpha}_{\iii|_{n}}^\jjj x \\
		&= t_{i_1}+\sum_{k=1}^{n-1} \overline{\alpha}_{\iii|_{k}}^\jjj(t_{i_{k+1}}-t_{i_k}) + \overline{\alpha}_{\iii|_{n}}^\jjj (x-t_{i_n}), \\
	\end{align*}
	we see that
  \begin{equation*}
    |\overline{f}_{\iii|_n}^\jjj(x)-t_{i_1}|\leq\sum_{k=1}^{n-1}\ualpha^k(t_N-t_1)+\ualpha^n\frac{t_N-t_1}{1-\ualpha}\leq\frac{\ualpha(t_N-t_1)}{1-\ualpha}
  \end{equation*}
  for all $x\in [t_1-\frac{\ualpha(t_N-t_1)}{1-\ualpha},t_N+\frac{\ualpha(t_N-t_1)}{1-\ualpha}]$ and for all $n \in \N$.
	Thus, we see that $\pi_{\jjj}(\iii) \leq t_N+\frac{\ualpha(t_N-t_1)}{1-\ualpha}$ for every realization. The lower bound follows similarly.
\end{proof}

By the previous lemma, the self-similar set with random contractions is contained in the interval $I$ defined in the lemma and is therefore uniformly bounded. The sequence $(\overline{f}_{\ell'\ell'\cdots|_n}(I))_{n \in \N}$ thus converges to $t_{\ell'}$ with uniform speed and there exists $p' \in \N$ such that
\begin{equation*}
  t_\ell \notin \overline{f}_{\ell'\ell'\cdots|_{p'}}^{\jjj}(I)
\end{equation*}
for all $\jjj \in \Sigma_*$ independently of the realization of the random variables. Furthermore, by \ref{it:ass3R}, the distribution of $\log \alpha_{\ell}$ is eventually smooth. Therefore, there is $p \in \N$ such that the $p$th self-convolution of the distribution is absolutely continuous with continuous density. This clearly implies that there exists a constant $C > 0$ such that
\begin{equation} \label{eq:dist}
  \mathbb{P}(\overline{\alpha}_{\ell\ell\cdots|_p} \in B(x,r))\leq\min\{1,Cr\}
\end{equation}
for all $x \in \R$ and $r>0$. Write $\kkk = \ell'\ell'\cdots|_{p'}$ and $\lll = \ell\ell\cdots|_p$.

\begin{lemma} \label{thm:tau1-far-away}
  If $X \subset \R$ is a self-similar set with random contractions, then, independently of the realization of the random variables, there exists a constant $c>0$ such that
  \begin{equation*}
    |\pi_{\jjj}(\iii)-t_\ell| \ge c
  \end{equation*}
  for all $\jjj\in\Sigma_*$ and $\iii \in [\kkk] \subset \Sigma$.
\end{lemma}

\begin{proof}
  Since $t_\ell \notin \overline{f}_\kkk^\jjj(I)$ independently of the realization of the random variables, the compact set $\pi_\jjj([\kkk]) = \overline{f}_{\jjj}^{-1} \circ \pi([\jjj\kkk]) \subseteq \overline{f}_\kkk^\jjj(I)$ is in bounded distance $c>0$ from $t_\ell$.
\end{proof}

Define $\Gamma_n = \JJ_n^\N \subset \Sigma$, where
\begin{equation*}
  \JJ_n = \{\iii\lll\kkk \in \Sigma_{n+p+p'} : \iii \in \Sigma_n \}
\end{equation*}
for all $n \in \N$. Note that for every $\iii,\jjj\in\Gamma_n$ with $\iii\neq\jjj$, there exists a unique $u\in\{1,\ldots,n\}$ such that $\sigma^{|\iii\wedge\jjj|+u}(\iii)|_{p+p'}=\lll\kkk \in \Sigma_{p+p'}$. Let $\mathcal{R}_{\iii,\jjj}$ be the $\sigma$-algebra generated by the random variables
\begin{equation*}
  \{\alpha_{\qqq} : \qqq\in\Sigma_*\}\setminus\{\alpha_{\iii|_{|\iii\wedge\jjj|+u+j}} : j \in \{1,\ldots,p+p'\}\}.
\end{equation*}
The following proposition is the transversality for self-similar sets with random contractions in the line.

\begin{proposition}\label{prop:trans}
  If $X \subset \R$ is a self-similar set with random contractions, then for every $n \in \N$ there exists a constant $C>0$ such that
  \begin{equation*}
    \mathbb{P}(|\pi(\iii)-\pi(\jjj)| < \roo \,|\, \mathcal{R}_{\iii,\jjj}) \le \min\Bigl\{1,C\frac{\roo}{|\overline{\alpha}_{\iii \land \jjj}|}\Bigr\}
  \end{equation*}
  for all $\roo>0$ and $\iii, \jjj \in \Gamma_n$ with $\iii \ne \jjj$.
\end{proposition}

\begin{proof}
  Fix $n \in \N$ and let $\iii,\jjj\in\Gamma_n$ with $\iii\neq\jjj$. For simplicity, write $\iii'=i_1'i_2'\cdots=\sigma^{|\iii\wedge\jjj|}\iii$ and $\jjj'=j_1'j_2'\cdots=\sigma^{|\iii\wedge\jjj|}\jjj$. Notice that $i_1'\neq j_1'$ and $\iii', \jjj' \in \bigcup_{k=0}^{n-1} \sigma^k(\Gamma_n)$. In particular, there exists a unique $k \in \{1,\ldots,n\}$ for which $\iii',\jjj' \in \sigma^{n-k}(\Gamma_n)$ and hence, $\sigma^k\iii'|_{p+p'} = \sigma^k\jjj'|_{p+p'}= \lll\kkk$. Since
  \begin{equation*}
    \pi_{\iii\wedge\jjj}(\iii') = \overline{f}_{\iii'|_k}^{\iii\wedge\jjj}(0) + \overline{\alpha}_{\iii'|_k}^{\iii\wedge\jjj}((1-\overline{\alpha}^{\iii\wedge\jjj\cdot\iii'|_k}_{\lll})t_\ell + \overline{\alpha}^{\iii\wedge\jjj\cdot\iii'|_k}_{\lll}\pi_{\iii\wedge\jjj\cdot\iii'|_k\cdot\lll}(\sigma^{k+p}\iii')),
  \end{equation*}
  we see that
  \begin{equation} \label{eq:trans1}
    \overline{\alpha}^{\iii\wedge\jjj\cdot\iii'|_k}_{\lll}(\pi_{\iii\wedge\jjj\cdot\iii'|_k\cdot\lll}(\sigma^{k+p}\iii')-t_\ell) = \frac{\pi_{\iii\wedge\jjj}(\iii') - \overline{f}_{\iii'|_k}^{\iii\wedge\jjj}(0)}{\overline{\alpha}_{\iii'|_k}^{\iii\wedge\jjj}} - t_\ell
  \end{equation}
  almost surely. Recall that, by \cref{thm:tau1-far-away}, there is a constant $c>0$ such that
  \begin{equation} \label{eq:trans2}
    |\pi_{\iii\wedge\jjj\cdot\iii'|_k}(\sigma^{k+p}\iii')-t_\ell| > c
  \end{equation}
  almost surely. By \cref{eq:defnatproj11}, we clearly have $\pi(\iii)=f_{i_1}(\pi_{i_1}(\sigma\iii))$, and more generally, $\pi(\iii)=\overline{f}_{\iii|_n}(\pi_{\iii|_n}(\sigma^n\iii))$ for all $n \in \N$. Hence,
  \begin{equation} \label{eq:trans3}
    \pi(\iii)-\pi(\jjj)=\overline{\alpha}_{\iii\wedge\jjj}(\pi_{\iii\wedge\jjj}(\sigma^{|\iii\wedge\jjj|}\iii)-\pi_{\iii\wedge\jjj}(\sigma^{|\iii\wedge\jjj|}\jjj))
  \end{equation}
  for all $\iii,\jjj \in \Sigma$. Since $|\alpha_\iii|\ge \lalpha$ for all $\iii \in \Sigma_*$, we thus see by recalling \cref{eq:trans3}, \cref{eq:trans1}, and \cref{eq:trans2} that
  \begin{align*}
    |\pi&(\iii)-\pi(\jjj)| < \roo \quad\Leftrightarrow\quad |\pi_{\iii\wedge\jjj}(\iii')-\pi_{\iii\wedge\jjj}(\jjj')| < \frac{\roo}{|\overline{\alpha}_{\iii\wedge\jjj}|} \\
    &\Leftrightarrow\quad \pi_{\iii\wedge\jjj}(\iii') \in B\Bigl(\pi_{\iii\wedge\jjj}(\jjj'),\frac{\roo}{|\overline{\alpha}_{\iii\wedge\jjj}|}\Bigr) \\
    &\Leftrightarrow\quad \overline{\alpha}^{\iii\wedge\jjj\cdot\iii'|_k}_{\lll}(\pi_{\iii\wedge\jjj\cdot\iii'|_k\cdot\lll}(\sigma^{k+m}\iii')-t_\ell) \in B\biggl( \frac{\pi_{\iii\wedge\jjj}(\jjj') - \overline{f}_{\iii'|_k}^{\iii\wedge\jjj}(0)}{\overline{\alpha}_{\iii'|_k}^{\iii\wedge\jjj}} - t_\ell,\frac{\roo}{|\overline{\alpha}_{\iii\wedge\jjj\cdot\iii'|_k}|} \biggr) \\
    &\Rightarrow\quad \overline{\alpha}^{\iii\wedge\jjj\cdot\iii'|_k}_{\lll} \in B\biggl(\frac{\pi_{\iii\wedge\jjj}(\iii') - \overline{f}_{\iii'|_k}^{\iii\wedge\jjj}(0)- t_\ell\overline{\alpha}_{\iii'|_k}^{\iii\wedge\jjj}}{\overline{\alpha}_{\iii'|_k}^{\iii\wedge\jjj}(\pi_{\iii\wedge\jjj\cdot\iii'|_k}(\sigma^{k+1}\iii')-t_\ell)},\frac{\roo}{|\overline{\alpha}_{\iii\wedge\jjj}|c\lalpha^n} \biggr)
  \end{align*}
  almost surely. Relying on the independence of the contraction ratios, the fact that $\overline{\alpha}^{\iii\wedge\jjj\cdot\iii'|_k}_{\lll}$ has the same distribution as $\overline{\alpha}_{\lll}$, and \cref{eq:dist}, we conclude
  \begin{equation*}
    \mathbb{P}(|\pi(\iii)-\pi(\jjj)| < \roo \,|\, \mathcal{R}_{\iii,\jjj}) \le \frac{2C}{c\lalpha^n} \frac{\roo}{|\overline{\alpha}_{\iii\wedge\jjj}|}
  \end{equation*}
  as claimed.
\end{proof}

\subsection{Transversality for random self-affine sponges} \label{sec:space}
We generalize the transversality proven for random self-similar sets in \cref{sec:line} to include random self-affine sponges. Let $\Pi(\Sigma) \subset \R^d$ be a self-affine sponge with random contractions, where
\begin{equation*}
  \Pi(\iii) = \sum_{n=1}^\infty A_{\iii|_0} \cdots A_{\iii|_{n-1}} (I-A_{\iii|_n})t_{i_n}.
\end{equation*}
Let $\proj_k \colon \R^d \to \R$, $\proj_k(x_1,\ldots,x_d) = x_k$, be the \emph{orthogonal projection} onto the $k$th coordinate axis. Define for each $\jjj \in \Sigma_*$ the \emph{canonical projection} $\pi_\jjj^k \colon \Sigma \to \R$ by setting
\begin{equation} \label{eq:defnatproj1}
  \pi_\jjj^k(\iii) = \sum_{n=1}^\infty \alpha_{\jjj\cdot\iii|_1}^{(k)} \cdots \alpha_{\jjj\cdot\iii|_{n-1}}^{(k)} (1-\alpha_{\jjj\cdot\iii|_n}^{(k)}) t_{i_n} \cdot e_k
\end{equation}
for all $\jjj \in \Sigma_*$ and $\iii \in \Sigma$. Writing $\pi^k = \pi_\varnothing^k$, we see that $\proj_k \circ \,\Pi = \pi^k$ for all $k \in \{1,\ldots,d\}$ and the coordinate functions of $\Pi = (\pi^1,\ldots,\pi^d)$ define a self-similar set $\pi^k(\Sigma) \subset \R$ with random contractions.

For each $k \in \{1,\ldots,d\}$ let $\ell_k, \ell_k' \in \{1,\ldots,N\}$ be given by \ref{it:ass3}. By \cref{lem:inint}, the self-similar set $\pi^k(\Sigma)$ with random contractions is uniformly bounded. Hence, by \cref{thm:tau1-far-away}, there exist $p_k' \in \N$ and a constant $c>0$ such that,
\begin{equation*}
  |\pi_\jjj^k(\kkk_k \iii)-t_{\ell_k}\cdot e_k| \ge c
\end{equation*}
for all $\jjj \in \Sigma_*$ and $\iii \in \Sigma$, where $\kkk_k = \ell_k'\cdots\ell_k'|_{p_k'}$. Since the distribution of $\log \alpha_{\ell_k}^{(k)}$ is eventually smooth, there is $p_k \in \N$ such that the $p_k$th self-convolution of the distribution is absolutely continuous with continuous density. Therefore, for every $k \in \{1,\ldots,d\}$ there exists a constant $C_k > 0$ such that
\begin{equation*}
  \mathbb{P}(\overline{\alpha}_{\lll_k} \in B(x,r))\leq\min\{1,C_kr\}
\end{equation*}
for all $x \in \R$ and $r>0$, where $\lll_k = \ell_k\ell_k\cdots|_{p_k}$. 

Write $p = p_1+\cdots+p_d$, $p' = p_1'+\cdots+p_d'$, and define $\Gamma_n = \JJ_n^\N \subset \Sigma$, where
\begin{equation*}
  \JJ_n = \{\iii\lll_1\kkk_1\cdots\lll_d\kkk_d \in \Sigma_{n+p+p'} : \iii \in \Sigma_n \}
\end{equation*}
for all $n \in \N$. Note that for every $\iii,\jjj\in\Gamma_n$ with $\iii\neq\jjj$, there exists a unique $u\in\{1,\ldots,n\}$ such that $\sigma^{|\iii\wedge\jjj|+u}(\iii)|_{p+p'}=\lll_1\kkk_1\cdots\lll_d\kkk_d \in \Sigma_{p+p'}$. Let $\mathcal{R}_{\iii,\jjj}$ be the $\sigma$-algebra generated by the random variables
\begin{equation*}
  \{\alpha_{\qqq}^{(k)} : \qqq\in\Sigma_* \text{ and } k\in\{1,\ldots,d\}\}\setminus\{\alpha_{\iii|_{|\iii\wedge\jjj|+u+j}}^{(k)} : j \in \{1,\ldots,p+p'\} \text{ and } k \in \{1,\ldots,d\}\}.
\end{equation*}
Similarly, let $\mathcal{R}_{\iii,\jjj}^{(k)}$ be the $\sigma$-algebra generated by the random variables
\begin{equation*}
  \{\alpha_{\qqq}^{(k)} : \qqq\in\Sigma_*\}\setminus\{\alpha_{\iii|_{|\iii\wedge\jjj|+u+j}}^{(k)} : j \in \{1,\ldots,p+p'\}\}
\end{equation*}
for all $k \in \{1,\ldots,d\}$. Note that the $\sigma$-algebras $\mathcal{R}_{\iii,\jjj}^{(1)},\ldots,\mathcal{R}_{\iii,\jjj}^{(d)}$ are clearly independent.
The following proposition is the transversality for self-affine sponges with random contractions.

\begin{proposition} \label{prop:transfin}
  If $X \subset \R^d$ is a self-affine sponge with random contractions, then for every $n \in \N$ there exists a constant $C>0$ such that
  \begin{equation*}
    \mathbb{P}(|\Pi(\iii)-\Pi(\jjj)| < \roo \,|\, \mathcal{R}_{\iii,\jjj}) \le \prod_{k=1}^d \min\Bigl\{1,C\frac{\roo}{|\overline{\alpha}_{\iii \land \jjj}^{(k)}|}\Bigr\}
  \end{equation*}
  for all $\roo>0$ and $\iii, \jjj \in \Gamma_n$ with $\iii \ne \jjj$.
\end{proposition}

\begin{proof}
  Notice first that if $|\Pi(\iii)-\Pi(\jjj)| < \roo$, then
  \begin{equation*}
    |\pi^k(\iii)-\pi^k(\jjj)| < \roo
  \end{equation*}
  for all $k \in \{1,\ldots,d\}$. Since the coordinates $\pi^k$ are independent, we thus have
  \begin{equation*}
    \mathbb{P}(|\Pi(\iii)-\Pi(\jjj)| < \roo \,|\, \mathcal{R}_{\iii,\jjj}) \le \prod_{k=1}^d \mathbb{P}(|\pi^k(\iii)-\pi^k(\jjj)| < \roo \,|\, \mathcal{R}_{\iii,\jjj}^{(k)}).
  \end{equation*}
  Observe that the definition of $\Gamma_n$ is more restrictive than the one introduced in \cref{sec:line}. As a result, the transversality condition in \cref{prop:trans} holds automatically for the coordinate sets $\pi^k(\Sigma)$. The claim then follows directly from \cref{prop:trans}.
\end{proof}

\section{Dimension of random self-affine sponges}

Using the tools developed in \cref{sec:measure} and \cref{sec:trans}, we investigate self-affine sponges $X \subset \R^d$ with random contractions. We prove the main result, \cref{thm:main}, by establishing the dimension upper bound in \cref{sec:upper-bound}, the lower bound in \cref{sec:lower-bound}, and the positivity of the Lebesgue measure in \cref{sec:positive-lebesgue}.

\subsection{Dimension upper bound} \label{sec:upper-bound}
If $X \subset \R^d$ is a self-affine sponge with random contractions, then we let $\varphi^s(\iii)=\varphi^s(A_{\iii})$ be as in \cref{eq:permutation-svf} and write
\begin{align*}
  \overline{\alpha}_\iii^{(k)} &= \alpha_{\iii|_1}^{(k)} \alpha_{\iii|_2}^{(k)} \cdots \alpha_{\iii|_{|\iii|-1}}^{(k)} \alpha_{\iii}^{(k)}, \\
  \overline{\fii}^s_\sigma(\iii) &= \fii^s_\sigma(\iii|_1) \cdots \fii^s_\sigma(\iii|_{|\iii|-1}) \fii^s_\sigma(\iii),
\end{align*}
for all $\iii \in \Sigma_*$ and $k \in \{1,\ldots,d\}$. We also define
\begin{equation} \label{eq:psi-permutation}
  \psi^s(\iii) = \max_{\sigma} \overline{\fii}^s_\sigma(\iii)
\end{equation}
for all $\iii \in \Sigma_*$ and $s \ge 0$. The upper bound for the dimension is given by the following lemma.

\begin{proposition} \label{lem:ub}
  If $X \subset \R^d$ is a self-affine sponge with random contractions, then we almost surely have
  \begin{equation*}
    \udimm(X) \le \min\{d,s_0\}.
  \end{equation*}
\end{proposition}

\begin{proof}
  We may assume without loss of generality that $s_0<d$. Fix $s \in (0,\infty) \setminus \N$ such that $s>s_0$ and let $\ell \in \N$ be such that $\ell-1 < s < \ell\leq d$. Let $\sigma_\iii$ denote the permutation which attains the maximum in \cref{eq:psi-permutation}. Observe that each set $\pi([\iii])$ can be covered by at most
  \begin{equation*}
   2\frac{|\overline{\alpha}_\iii^{(\sigma_\iii(1))}|}{|\overline{\alpha}_\iii^{(\sigma_\iii(\ell))}|} \cdots \frac{|\overline{\alpha}_\iii^{(\sigma_\iii(\ell-1))}|}{|\overline{\alpha}_\iii^{(\sigma_\iii(\ell))}|}
  \end{equation*}
  many balls of radius $\sqrt{d}|\overline{\alpha}_\iii^{(\sigma_\iii(\ell))}|$. Let $n \in \N$ and choose an integer $m_n \le n$ such that $\lalpha^{m_n}\leq \ualpha^n<\lalpha^{m_n-1}$. Define
  \begin{equation*}
    C_n=\{\iii\in\Sigma_*:|\overline{\alpha}_{\iii}^{(\sigma_\iii(\ell))}|\leq\ualpha^n<|\overline{\alpha}_{\iii|_k}^{(\sigma_{\iii|_k}(\ell))}|\text{ for all }k \in \{1,\ldots,|\iii|-1\}\}
  \end{equation*}
  and observe that $C_n$ is a random cover for $\Sigma$ such that $|\iii|\in\{m_n,\ldots,n\}$ for all $\iii \in C_n$. Thus, each set $\pi([\iii])$, $\iii \in C_n$, can be covered by at most
  \begin{equation*}
  	2\frac{|\overline{\alpha}_\iii^{(\sigma_\iii(1))}|}{|\overline{\alpha}_\iii^{(\sigma_\iii(\ell))}|} \cdots \frac{|\overline{\alpha}_\iii^{(\sigma_\iii(\ell-1))}|}{|\overline{\alpha}_\iii^{(\sigma_\iii(\ell))}|}
  \end{equation*}
  many balls with radius $\sqrt{d}\ualpha^n$, and so,
  \begin{equation*}
    N_{\sqrt{d}\ualpha^n}(X)\leq 2\sum_{\iii\in C_n}\frac{|\overline{\alpha}_\iii^{(\sigma_\iii(1))}|}{|\overline{\alpha}_\iii^{(\sigma_\iii(\ell))}|} \cdots \frac{|\overline{\alpha}_\iii^{(\sigma_\iii(\ell-1))}|}{|\overline{\alpha}_\iii^{(\sigma_\iii(\ell))}|}.
  \end{equation*}
  Let $A_n$ denote the upper bound defined above. Since $|\overline{\alpha}_{\iii}^{(\sigma_\iii(\ell))}|\le \ualpha^n$ for $\iii\in C_n$, we have
  \begin{equation*}
    A_n\ualpha^{ns} \le 2\sum_{\iii\in C_n}\psi^s(\iii).
  \end{equation*}
  Our objective is to prove that
  \begin{equation*}
    \mathbb{P}( A_n > \ualpha^{-ns} \text{ for infinitely many } n ) = 0.
  \end{equation*}
  Since $s > s_0$ is arbitrary, this establishes the proposition, as it implies the desired decay of $A_n$ in \cref{eq:def-minkowski} for all $s > s_0$. By the Borel-Cantelli lemma, it is enough to show that $\sum_{n=1}^\infty\mathbb{P}(A_n>\ualpha^{-ns})<\infty$. By the independence \ref{it:ass1} and the identical distribution \ref{it:ass1b}, we see that
  \begin{align*}
  	\mathbb{P}(A_n\ualpha^{ns}>1)&\le\mathbb{P}\biggl(\frac12<\sum_{\iii\in C_n}\psi^s(\iii)\biggr)\le\mathbb{P}\biggl(\frac12<\sum_{k=m_n}^n\sum_{\iii\in\Sigma_k}\psi^s(\iii)\biggr)\\ 
    &\le 2\,\mathbb{E}\biggl( \sum_{k=m_n}^n\sum_{\iii\in\Sigma_k} \psi^s(\iii) \biggr)\le \ualpha^{m_n(s-s_0)}\sum_{k=m_n}^n\mathbb{E}\biggl( \sum_{\iii \in \Sigma_k} \psi^{s_0}(\iii) \biggr)\\ 
    &\leq \ualpha^{m_n(s-s_0)}\sum_{k=m_n}^n\mathbb{E}\biggl(\sum_{\iii \in \Sigma_k} \sum_\sigma \overline{\fii}_\sigma^{s_0}(\iii)\biggr) \\ 
    &= \ualpha^{m_n(s-s_0)}\sum_{k=m_n}^n \sum_\sigma\biggl( \sum_{i=1}^N \mathbb{E}(\fii_\sigma^{s_0}(i)) \biggr)^k \\
    &\le \ualpha^{m_n(s-s_0)} d!\sum_{k=m_n}^n \biggl( \max_{\sigma} \sum_{i=1}^N \mathbb{E}(\fii_\sigma^{s_0}(i)) \biggr)^k=n\ualpha^{m_n(s-s_0)} d!,
  \end{align*}
  where $\sum_\sigma$ denotes the sum over all permutations $\sigma$. Since this sequence forms a convergent series, the claim follows. 
\end{proof}

\subsection{Dimension lower bound} \label{sec:lower-bound}
By the independence \ref{it:ass1} and the identical distribution \ref{it:ass1b}, we see that
\begin{equation*}
  \sum_{\iii\in \JJ_n}\mathbb{E}(\overline{\fii}^s_\sigma(\iii)) = \biggl(\sum_{i=1}^N\mathbb{E}(\fii^s_\sigma(i))\biggr)^n \prod_{k=1}^d\mathbb{E}(\fii_\sigma^s(\ell_k'))^{p_k'}\mathbb{E}(\fii_\sigma^s(\ell_k))^{p_k}.
\end{equation*}
If $s_n$ is the unique solution of
\begin{equation} \label{eq:approx}
  \max_{\sigma}\sum_{\iii\in \JJ_n}\mathbb{E}(\overline{\fii}^s_\sigma(\iii))=1,
\end{equation}
then we clearly have $s_n\to s_0$ as $n\to\infty$. Let $\sigma$ be a permutation for which the maximum in \cref{eq:approx} is attained and write $q_n = n+p+p'$ to denote the length of elements in $\JJ_n$. For each $\iii \in \bigcup_{k=1}^\infty \JJ_n^k$ we set
\begin{equation*}
  p_{\iii} = \frac{\overline{\fii}^{s_n}_\sigma(\iii)}{\overline{\fii}^{s_n}_\sigma(\iii|_{|\iii|-q_n})} = \fii^{s_n}_\sigma(\iii|_{|\iii|-q_n+1})\cdots\fii^{s_n}_\sigma(\iii|_{|\iii|-1})\fii^{s_n}_\sigma(\iii).
\end{equation*}
It is straightforward to see that $\{p_{\iii} : \iii \in \bigcup_{k=1}^\infty \JJ_n^k\}$ satisfies the properties \ref{it:dist}--\ref{it:expect}. Let $\mu_n$ be the random measure given by \cref{thm:measure} supported on $\Gamma_n$. The measure $\mu_n$ is defined by
\begin{equation} \label{eq:def-of-meas}
\begin{split}
  \mu_n([\iii]) &= p_{\iii|_{q_n}}p_{\iii|_{2q_n}}\cdots p_{\iii|_{kq_n}} X^\iii \\
  &= \overline{\fii}^{s_n}_\sigma(\iii|_{q_n}) \frac{\overline{\fii}^{s_n}_\sigma(\iii|_{2q_n})}{\overline{\fii}^{s_n}_\sigma(\iii|_{2q_n-q_n})} \cdots \frac{\overline{\fii}^{s_n}_\sigma(\iii|_{kq_n})}{\overline{\fii}^{s_n}_\sigma(\iii|_{kq_n-q_n})}X^\iii = \overline{\fii}_\sigma^{s_n}(\iii)X^{\iii}
\end{split}
\end{equation}
for all $\iii \in \JJ_n^k$ and $k \in \N$, where the $L^2$-random variable $X^\iii$ is introduced in the proof of \cref{thm:measure}. Furthermore, $\mu_n$ can naturally be extended to $\Sigma$ as well as $X^\iii$ can be extended to $\Sigma_*$. Let $\mathcal{F}_n$ be the $\sigma$-algebra generated by the random variables $\{\alpha_{\iii}^{(k)} : |\iii| \leq n \text{ and } k\in\{1,\ldots,d\}\}$. Observe that $X^{\iii}$ is independent of $\mathcal{F}_{|\iii|}$.

Recall that, by \cref{eq:defnatproj1}, we have
\begin{equation*}
  \pi_\jjj^k(\iii)
  = t_{i_1}\cdot e_k + \sum_{j=1}^{\infty} \alpha_{\jjj\cdot\iii|_1}^{(k)} \cdots \alpha_{\jjj\cdot\iii|_j}^{(k)} (t_{i_{j+1}}-t_{i_j}) \cdot e_k
\end{equation*}
for all $\jjj \in \Sigma_*$ and $\iii \in \Sigma$. 
Therefore, as in \cref{eq:trans3},
\begin{align*}
  \pi_\jjj^k(\iii) - \pi_\jjj^k(\iii|_n \cdot\ell\ell\cdots)
  &= \alpha_{\jjj\cdot\iii|_1}^{(k)} \cdots \alpha_{\jjj\cdot\iii|_{n-1}}^{(k)} (\pi_{\jjj\cdot\iii|_n}^k(\sigma^n\iii)-\pi_{\jjj\cdot\iii|_n}^k(\ell\ell\cdots)) \\
  &= \alpha_{\jjj\cdot\iii|_1}^{(k)} \cdots \alpha_{\jjj\cdot\iii|_{n-1}}^{(k)} (\pi_{\jjj\cdot\iii|_n}^k(\sigma^n\iii)-t_{\ell}\cdot e_k).
\end{align*}
Let $\conv(\pi^k(\Sigma))$ be the convex hull of $\pi^k(\Sigma)$; recall \cref{lem:inint}. We thus have
\begin{equation} \label{eq:approx2}
  |\pi_\jjj^k(\iii) - \pi_\jjj^k(\iii|_n \cdot\ell\ell\cdots)| \le \ualpha^n \diam(\conv(\pi^k(\Sigma)))
\end{equation}

\begin{proposition}\label{lem:lb}
  If $X \subset \R^d$ is a self-affine sponge with random contractions, then we almost surely have
  \begin{equation*}
    \dimh(X) \ge \min\{d,s_0\}.
  \end{equation*}
\end{proposition}

\begin{proof}
  Let $s_n$ be the unique solution of \cref{eq:approx}. It suffices to show that for each $n$ we almost surely have $\dimh(X) \ge \min\{d,s_n\}$. Fix $n \in \N$ and let $t < \min\{d,s_n\}$. Writing $\Xi_b=\{\iii\in\Sigma_b : [\iii]\cap\Gamma_n\neq\emptyset\}$, we see that
  \begin{align*}
    \mathbb{E}\biggl( &\iint \frac{1}{|\Pi(\iii)-\Pi(\jjj)|^t} \dd\mu_n(\iii) \dd\mu_n(\jjj) \biggr) \\
    &= \sum_{b=0}^\infty \sum_{\hhh \in \Xi_b} \sum_{i \ne j} \int_0^\infty \mathbb{E}\biggl( \iint_{[\hhh i] \times [\hhh j]} \mathds{1}\{|\Pi(\iii)-\Pi(\jjj)| < \roo^{-1/t}\} \dd\mu_n(\iii) \dd\mu_n(\jjj) \biggr) \dd\roo.
  \end{align*}
  If $(\iii,\jjj) \in [\hhh i] \times [\hhh j]$, then we see that $|\Pi(\iii)-\Pi(\jjj)| < \roo^{-1/t}$ if and only if
  \begin{equation*}
    |\pi^k_\hhh(\sigma^b\iii)-\pi^k_\hhh(\sigma^b\jjj)| < \frac{\roo^{-1/t}}{|\overline{\alpha}_\hhh^{(k)}|}
  \end{equation*}
  for all $k \in \{1,\ldots,d\}$. Let $m=m(b,\roo)$ be such that
  \begin{align*}
    m=\min\{u\geq q_n+1 : \;&\text{$u+b$ is divisible by $q_n$ and } \\
    &\ualpha^{u+b}\max_k\diam(\conv(\pi^k(\Sigma)))\leq\roo^{-1/t}\}.
  \end{align*}
  Then, by \cref{eq:approx2}, for $\iii\in[\hhh]$ with $|\hhh|=b$ we have
  \begin{equation*}
    |\pi_{\hhh}^k(\sigma^b\iii)-\pi_{\hhh}^k(\sigma^b\iii|_m\cdot\ell\ell\cdots)| \leq \ualpha^m\diam(\conv(\pi^k(\Sigma))) \leq \frac{\roo^{-1/t}}{\ualpha^b} \leq \frac{\roo^{-1/t}}{|\overline{\alpha}_{\hhh}^{(k)}|}.
  \end{equation*}
  Since $|\overline{\alpha}_{\hhh}^{(k)}|\leq \ualpha^b$ for all $k$, the last inequality holds uniformly in $b$. Thus, taking $(\iii,\jjj)\in[\hhh i\ppp]\times[\hhh j\qqq]$, where $|\ppp|=|\qqq|=m$ and $i\neq j$, we have
  \begin{equation*}
    \{|\Pi(\iii)-\Pi(\jjj)| < \roo^{-1/t}\}\subseteq\bigcap_{k=1}^d\biggl\{|\pi_{\hhh}^k(i\ppp)-\pi_{\hhh}^k(j\qqq)| < \frac{3\roo^{-1/t}}{|\overline{\alpha}_\hhh^{(k)}|}\biggr\}.
  \end{equation*}
  Hence,
  \begin{align*}
  	\mathbb{E}\biggl( &\iint \frac{1}{|\Pi(\iii)-\Pi(\jjj)|^t} \dd\mu_n(\iii) \dd\mu_n(\jjj) \biggr) \\
  	&= \sum_{b=0}^\infty \sum_{\hhh \in \Xi_b} \sum_{i \ne j} \int_0^\infty \mathbb{E}\biggl( \iint_{[\hhh i] \times [\hhh j]} \mathds{1}\{|\Pi(\iii)-\Pi(\jjj)| < \roo^{-1/t}\} \dd\mu_n(\iii) \dd\mu_n(\jjj) \biggr) \dd\roo,\\
  	&\leq \sum_{b=0}^\infty \sum_{\hhh \in \Xi_b} \sum_{i \ne j} \int_0^\infty\sum_{\ppp,\qqq\in\Sigma_m} \mathbb{E}\biggl( \mathds{1}\bigcap_{k=1}^d\biggl\{|\pi_{\hhh}^k(i\ppp)-\pi_{\hhh}^k(j\qqq)| < \frac{3\roo^{-1/t}}{|\overline{\alpha}_\hhh^{(k)}|}\}\\
  	&\qquad\qquad\qquad\qquad\qquad\qquad\qquad\qquad\qquad\qquad\qquad\mu_n([\hhh i\ppp])\mu_n([\hhh j\qqq]) \biggr) \dd\roo.
  \end{align*}
  By \cref{eq:def-of-meas}, we have
  \begin{equation*}
    \mu_n([\hhh i\ppp])\mu_n([\hhh j\qqq])=\overline{\fii}_\sigma^{s_n}(\hhh i\ppp)\overline{\fii}_\sigma^{s_n}(\hhh j\qqq)X^{\hhh j\qqq}X^{\hhh i\ppp}.
  \end{equation*}
  The random variables $X^{\hhh j\qqq}$ and $X^{\hhh i\ppp}$ are independent, and also independent of the $\sigma$-algebra $\mathcal{F}_{b+m+1}$, since each $X^{\iii}$ depends only on random variables indexed by descendants of $\iii$. Furthermore, the random variable
  \begin{equation*}
    \mathds{1}\bigcap_{k=1}^d\biggl\{|\pi_{\hhh}^k(i\ppp)-\pi_{\hhh}^k(j\qqq)| < \frac{3\roo^{-1/t}}{|\overline{\alpha}_\hhh^{(k)}|}\biggr\}\overline{\fii}_\sigma^{s_n}(\hhh i\ppp)\overline{\fii}_\sigma^{s_n}(\hhh j\qqq)
  \end{equation*}
  is $\mathcal{F}_{b+m+1}$-measurable. Hence,
  \begin{align*}
  	 \mathbb{E}\biggl(&\mathds{1}\bigcap_{k=1}^d\biggl\{|\pi_{\hhh}^k(i\ppp)-\pi_{\hhh}^k(j\qqq)| < \frac{3\roo^{-1/t}}{|\overline{\alpha}_\hhh^{(k)}|}\biggr\}\mu_n([\hhh i\ppp])\mu_n([\hhh j\qqq]) \biggr)\\
  	 &=\mathbb{E}\biggl(\mathbb{E}\biggl( \mathds{1}\bigcap_{k=1}^d\biggl\{|\pi_{\hhh}^k(i\ppp)-\pi_{\hhh}^k(j\qqq)| < \frac{3\roo^{-1/t}}{|\overline{\alpha}_\hhh^{(k)}|}\biggr\}\mu_n([\hhh i\ppp])\mu_n([\hhh j\qqq]) \,\Big|\,\mathcal{F}_{b+m+1}\biggr)\biggl)\\
  	 &=\mathbb{E}\biggl(\mathds{1}\bigcap_{k=1}^d\biggl\{|\pi_{\hhh}^k(i\ppp)-\pi_{\hhh}^k(j\qqq)| < \frac{3\roo^{-1/t}}{|\overline{\alpha}_\hhh^{(k)}|}\biggr\}\overline{\fii}_\sigma^{s_n}(\hhh i\ppp)\overline{\fii}_\sigma^{s_n}(\hhh j\qqq)\mathbb{E}\biggl( X^{\hhh j\qqq}X^{\hhh i\ppp} \,\Big|\,\mathcal{F}_{b+m+1}\biggr)\biggl)\\
  	 &=\mathbb{E}\biggl(\mathds{1}\bigcap_{k=1}^d\biggl\{|\pi_{\hhh}^k(i\ppp)-\pi_{\hhh}^k(j\qqq)| < \frac{3\roo^{-1/t}}{|\overline{\alpha}_\hhh^{(k)}|}\biggr\}\overline{\fii}_\sigma^{s_n}(\hhh i\ppp)\overline{\fii}_\sigma^{s_n}(\hhh j\qqq)\biggl).
  \end{align*}
  Let $u=u(i,j)\in\{1,\ldots,n\}$ be the unique integer such that $\sigma^{u}(\ppp)|_{p+p'}=\lll_1\kkk_1\cdots\lll_d\kkk_d \in \Sigma_{p+p'}$. Set
  \begin{equation*}
    W_{\hhh,i,\ppp}=\overline{\fii}_\sigma^{s_n}(\hhh i\cdot\ppp|_u)\frac{\overline{\fii}_\sigma^{s_n}(\hhh i\ppp)}{\overline{\fii}_\sigma^{s_n}(\hhh i\cdot\ppp|_{u+p+p'})}.
  \end{equation*}
  Since the block $\ppp|_{u+p+p'}$ has length $p+p'$, we have
  \begin{equation*}
  	\lalpha^{s_n(p+p')} W_{\hhh,i,\ppp} \le \overline{\fii}_\sigma^{s_n}(\hhh i\ppp) \le \ualpha^{s_n(p+p')} W_{\hhh,i,\ppp}.
  \end{equation*}
  Since $W_{\hhh,i,\ppp}$ and $\overline{\fii}_\sigma^{s_n}(\hhh j\qqq)$ are $\mathcal{R}_{\hhh,i\ppp}$-measurable, \cref{prop:transfin} gives
  \begin{align*}
    \mathbb{E}\biggl(&\mathds{1}\bigcap_{k=1}^d\biggl\{|\pi_{\hhh}^k(i\cdot\ppp)-\pi_{\hhh}^k(j\cdot\qqq)| < \frac{3\roo^{-1/t}}{|\overline{\alpha}_\hhh^{(k)}|}\biggr\}\overline{\fii}_\sigma^{s_n}(\hhh i\ppp)\overline{\fii}_\sigma^{s_n}(\hhh j\qqq)\biggl)\\
    &=\mathbb{E}\biggl(\mathbb{E}\biggl(\mathds{1}\bigcap_{k=1}^d\biggl\{|\pi_{\hhh}^k(i\ppp)-\pi_{\hhh}^k(j\qqq)| < \frac{3\roo^{-1/t}}{|\overline{\alpha}_\hhh^{(k)}|}\biggr\}\overline{\fii}_\sigma^{s_n}(\hhh i\ppp)\overline{\fii}_\sigma^{s_n}(\hhh j\qqq)\,\Big|\,\mathcal{R}_{\hhh,i\ppp}\biggr)\biggr)\\
    &\leq\Bigl(\frac{\ualpha}{\lalpha}\Bigr)^{s_n(p+p')}\mathbb{E}\biggl(\overline{\fii}_\sigma^{s_n}(\hhh i\ppp)\overline{\fii}_\sigma^{s_n}(\hhh j\qqq)\mathbb{P}\biggl(\bigcap_{k=1}^d\biggl\{|\pi_{\hhh}^k(i\ppp)-\pi_{\hhh}^k(j\qqq)| < \frac{3\roo^{-1/t}}{|\overline{\alpha}_\hhh^{(k)}|}\biggr\}\,\Big|\, \mathcal{R}_{\hhh,i\ppp}\biggr)\biggr)\\
    &\leq\Bigl(\frac{\ualpha}{\lalpha}\Bigr)^{s_n(p+p')}\mathbb{E}\biggl(\overline{\fii}_\sigma^{s_n}(\hhh i\ppp)\overline{\fii}_\sigma^{s_n}(\hhh j\qqq)\prod_{k=1}^d\min\biggl\{1, C\frac{3\roo^{-1/t}}{|\overline{\alpha}_\hhh^{(k)}|}\biggr\}\biggr).
  \end{align*}
  Using this, we get
  \begin{align*}
  	\mathbb{E}\biggl( &\iint \frac{1}{|\Pi(\iii)-\Pi(\jjj)|^t} \dd\mu_n(\iii) \dd\mu_n(\jjj) \biggr) \\
  	&\leq \sum_{b=0}^\infty \sum_{\hhh \in \Xi_b} \sum_{i \ne j} \int_0^\infty\sum_{\ppp,\qqq\in\Sigma_m} \mathbb{E}\biggl( \mathds{1}\bigcap_{k=1}^d\biggl\{|\pi_{\hhh}^k(i\ppp)-\pi_{\hhh}^k(j\qqq)| < \frac{3\roo^{-1/t}}{|\overline{\alpha}_\hhh^{(k)}|}\biggr\}\\
  	&\qquad\qquad\qquad\qquad\qquad\qquad\qquad\qquad\qquad\qquad\quad\mu_n([\hhh i\ppp])\mu_n([\hhh j\qqq]) \biggr) \dd\roo\\
  	&\leq\Bigl(\frac{\ualpha}{\lalpha}\Bigr)^{s_n(p+p')}\sum_{b=0}^\infty \sum_{\hhh \in \Xi_b} \int_0^\infty\sum_{\ppp,\qqq\in\Sigma_m}\mathbb{E}\biggl(\overline{\fii}_\sigma^{s_n}(\hhh i\ppp)\overline{\fii}_\sigma^{s_n}(\hhh j\qqq)\prod_{k=1}^d\min\biggl\{1, C\frac{3\roo^{-1/t}}{|\overline{\alpha}_\hhh^{(k)}|}\biggr\}\biggr)\dd\roo\\
  	&=\Bigl(\frac{\ualpha}{\lalpha}\Bigr)^{s_n(p+p')}\sum_{b=0}^\infty \sum_{\hhh \in \Xi_b} \int_0^\infty\mathbb{E}\biggl(\overline{\fii}_\sigma^{s_n}(\hhh)^2\prod_{k=1}^d\min\biggl\{1, C\frac{3\roo^{-1/t}}{|\overline{\alpha}_\hhh^{(k)}|}\biggr\}\biggr)\dd\roo,
  \end{align*}
  where we used the definition of $s_n$. Since $t<d$, there is a constant $C_t>0$, depending only on $t$ and $d$, such that
  \begin{equation*}
    \int_0^\infty\prod_{k=1}^d\min\biggl\{1, \frac{\roo^{-1/t}}{|\overline{\alpha}_\hhh^{(k)}|}\biggr\}\dd\roo \le C_t\,\psi^t(\hhh)^{-1}.
  \end{equation*}
  Hence,
  \begin{align*}
    \mathbb{E}\biggl( &\iint \frac{1}{|\Pi(\iii)-\Pi(\jjj)|^t} \dd\mu_n(\iii) \dd\mu_n(\jjj) \biggr) \\
  	&\leq\Bigl(\frac{\ualpha}{\lalpha}\Bigr)^{s_n(p+p')}(3C)^d C_t\sum_{b=0}^\infty \sum_{\hhh \in \Xi_b}  \mathbb{E}(\overline{\fii}_\sigma^{s_n}(\hhh)^2\psi^t(\hhh)^{-1})\\
  	&\leq\Bigl(\frac{\ualpha}{\lalpha}\Bigr)^{s_n(p+p')}(3C)^d\sum_{b=0}^\infty \sum_{\hhh \in \Xi_b}  \mathbb{E}(\psi^{s_n}(\hhh)^2\psi^t(\hhh)^{-1})\\
  	&\leq\Bigl(\frac{\ualpha}{\lalpha}\Bigr)^{s_n(p+p')}(3C)^d\sum_{b=0}^\infty \ualpha^{(s_n-t)b}\sum_{\hhh \in \Xi_b}  \mathbb{E}(\psi^{s_n}(\hhh))\\
  	&\leq\Bigl(\frac{\ualpha}{\lalpha}\Bigr)^{s_n(p+p')}(3C)^d\sum_{b=0}^\infty \ualpha^{(s_n-t)b}\sum_{\sigma}\sum_{\hhh \in \Xi_b} \mathbb{E}(\overline\fii^{s_n}_\sigma(\hhh))\\
  	&\leq\Bigl(\frac{\ualpha}{\lalpha}\Bigr)^{s_n(p+p')}(3C)^d\sum_{b=0}^\infty \ualpha^{(s_n-t)b}d!<\infty
  \end{align*}
  which, by \cite[Theorem 8.7]{Mattila1995}, yields $\dimh(X) \ge t$ almost surely as wished.
\end{proof}

\subsection{Positive Lebesgue measure} \label{sec:positive-lebesgue}
The proposition below demonstrates the final claim of Theorem~\ref{thm:main}.

\begin{proposition} \label{thm:positive-lebesgue}
	If $X \subset \R^d$ is a self-affine sponge with random contractions such that $s_0>d$, then we almost surely have $\LL^d(X)>0$.
\end{proposition}

\begin{proof}
  Let $s_n$ be the unique solution of \cref{eq:approx}. Since $s_n$ converges to $s_0$, we may fix $n \in \N$ such that $s_n>d$. Let $\mu_n$ be the random measure defined in \cref{eq:def-of-meas}. It is enough to show that $\Pi_*\mu_n\ll\LL^d$ almost surely. To that end, let
  \begin{equation*}
    \underline{D}(\Pi_*\mu_n,x)=\liminf_{r\to0}\frac{\Pi_*\mu_n(B(x,r))}{\LL^d(B(x,r))}
  \end{equation*}
  be the lower local density of $\Pi_*\mu_n$ at $x$. By \cite[Theorem 2.12(iii)]{Mattila1995}, it suffices to show that
  \begin{equation*}
    \mathbb{E}\biggl(\int\underline{D}(\Pi_*\mu_n,x)\dd\Pi_*\mu_n(x)\biggr)<\infty.
  \end{equation*}
 	Writing $\Xi_b=\{\iii\in\Sigma_b : [\iii]\cap\Gamma_n\neq\emptyset\}$, we see by Fatou's lemma that
	\begin{align*}
	 	\mathbb{E}\biggl(&\int\underline{D}(\Pi_*\mu_n,x)\dd\Pi_*\mu_n(x)\biggr)\\
	 	&\leq\liminf_{r\to0}\frac{1}{r^d\LL^d(B(0,1))}\sum_{b=0}^\infty \sum_{\hhh \in \Xi_b} \sum_{i \ne j}\mathbb{E}\biggl(\iint_{[\hhh i] \times [\hhh j]}\mathds{1}\{|\Pi(\iii)-\Pi(\jjj)|<r\}\\
    &\qquad\qquad\qquad\qquad\qquad\qquad\qquad\qquad\qquad\qquad\qquad\qquad\quad\dd\mu_n(\iii)\dd\mu_n(\jjj)\biggr)
	\end{align*}
  If $(\iii,\jjj) \in [\hhh i] \times [\hhh j]$, then, as in the proof of \cref{prop:transfin}, we see that $|\Pi(\iii)-\Pi(\jjj)| < r$ if and only if
  \begin{equation*}
 	  |\pi^k_\hhh(\sigma^b\iii)-\pi^k_\hhh(\sigma^b\jjj)| < \frac{r}{|\overline{\alpha}_\hhh^{(k)}|}
  \end{equation*}
  for all $k \in \{1,\ldots,d\}$. Let $m=m(b,r)$ be such that
  \begin{align*}
 	  m=\min\{u\geq q_n+1 : \;&\text{$u+b$ is divisible by $q_n$ and } \\
 	  &\ualpha^{u+b}\max_k\diam(\conv(\pi^k(\Sigma)))\leq r\}.
  \end{align*}
  Then, by \cref{eq:approx2}, for $\iii\in[\hhh]$ with $|\hhh|=b$ we have
  \begin{equation*}
 	  |\pi_{\hhh}^k(\sigma^b\iii)-\pi_{\hhh}^k(\sigma^b\iii|_m\cdot\ell\ell\cdots)| \leq \frac{r}{\ualpha^b} \leq \frac{r}{|\overline{\alpha}_\hhh^{(k)}|}.
  \end{equation*}
  As above, this uses $|\overline{\alpha}_\hhh^{(k)}|\leq \ualpha^b$ for all $k$. Thus, taking $(\iii,\jjj)\in[\hhh i\ppp]\times[\hhh j\qqq]$, where $|\ppp|=|\qqq|=m$ and $i\neq j$, we have
  \begin{equation*}
 	  \{|\Pi(\iii)-\Pi(\jjj)| < r\}\subseteq\bigcap_{k=1}^d\biggl\{|\pi_{\hhh}^k(i\ppp)-\pi_{\hhh}^k(j\qqq)| < \frac{3r}{|\overline{\alpha}_\hhh^{(k)}|}\biggr\}.
  \end{equation*}
  Hence,
  \begin{align*}
	  \liminf_{r\to0}&\frac{1}{r^d\LL^d(B(0,1))}\sum_{b=0}^\infty \sum_{\hhh \in \Xi_b} \sum_{i \ne j}\\
    &\qquad\quad\mathbb{E}\biggl(\iint_{[\hhh i] \times [\hhh j]}\mathds{1}\{|\Pi(\iii)-\Pi(\jjj)|<r\}\dd\mu_n(\iii)\dd\mu_n(\jjj)\biggr)\\
	  &\leq\liminf_{r\to0}\frac{1}{r^d\LL^d(B(0,1))}\sum_{b=0}^\infty \sum_{\hhh \in \Xi_b} \sum_{i \ne j}\sum_{\ppp,\qqq\in\Sigma_m}\\
    &\qquad\qquad\qquad\quad\;\mathbb{E}\biggl(\mathds{1}\bigcap_{k=1}^d\biggl\{|\pi_{\hhh}^k(i\ppp)-\pi_{\hhh}^k(j\qqq)| < \frac{3r}{|\overline{\alpha}_\hhh^{(k)}|}\biggr\}\mu_n([\hhh i\ppp])\mu_n([\hhh j\qqq])\biggr).
  \end{align*}
  Proceeding as in the proof of \cref{lem:lb}, conditioning on $\mathcal{F}_{b+m+1}$ and using the independence of the variables $X^{\iii}$ (each $X^{\iii}$ depends only on descendants of $\iii$), we get
  \begin{align*}
	  \mathbb{E}\biggl(&\mathds{1}\bigcap_{k=1}^d\biggl\{|\pi_{\hhh}^k(i\ppp)-\pi_{\hhh}^k(j\qqq)| < \frac{3r}{|\overline{\alpha}_\hhh^{(k)}|}\biggr\}\mu_n([\hhh i\ppp])\mu_n([\hhh j\qqq])\biggr)\\
	  &=\mathbb{E}\biggl(\mathds{1}\bigcap_{k=1}^d\biggl\{|\pi_{\hhh}^k(i\ppp)-\pi_{\hhh}^k(j\qqq)| < \frac{3r}{|\overline{\alpha}_\hhh^{(k)}|}\biggr\}\overline{\fii}_\sigma^{s_n}(\hhh i\ppp)\overline{\fii}_\sigma^{s_n}(\hhh j\qqq)\biggr)\\
	  &\leq\Bigl(\frac{\ualpha}{\lalpha}\Bigr)^{s_n(p+p')}\mathbb{E}\biggl(\overline{\fii}_\sigma^{s_n}(\hhh i\ppp)\overline{\fii}_\sigma^{s_n}(\hhh j\qqq)\prod_{k=1}^d\min\biggl\{1, C\frac{3r}{|\overline{\alpha}_\hhh^{(k)}|}\biggr\}\biggr)\\
	  &\leq C^d3^dr^d\Bigl(\frac{\ualpha}{\lalpha}\Bigr)^{s_n(p+p')}\mathbb{E}(\overline{\fii}_\sigma^{s_n}(\hhh i\ppp)\overline{\fii}_\sigma^{s_n}(\hhh j\qqq)\overline{\fii}_\sigma^{d}(\hhh)^{-1}),
  \end{align*}
  where we note that since $s_n>d$, the random function $\overline{\fii}_\sigma^{s_n}$ is independent of the choice of $\sigma$ and $\overline{\fii}_\sigma^{s_n}(\hhh)=\bigl|\prod_{i=1}^d\overline{\alpha}_\hhh^{(i)}\bigr|^{s_n/d}$. Therefore, we have
  \begin{align*}
	  \mathbb{E}\biggl(&\int\underline{D}(\Pi_*\mu_n,x)\dd\Pi_*\mu_n(x)\biggr)\\
	  &\leq\Bigl(\frac{\ualpha}{\lalpha}\Bigr)^{s_n(p+p')}\frac{C^d3^d}{\LL^d(B(0,1))}\sum_{b=0}^\infty \sum_{\hhh \in \Xi_b} \sum_{i \ne j}\sum_{\ppp,\qqq\in\Sigma_m}\mathbb{E}(\overline{\fii}_\sigma^{s_n}(\hhh i\ppp)\overline{\fii}_\sigma^{s_n}(\hhh j\qqq)\overline{\fii}_\sigma^{d}(\hhh)^{-1})\\
	  &\leq\Bigl(\frac{\ualpha}{\lalpha}\Bigr)^{s_n(p+p')}\frac{C^d3^d}{\LL^d(B(0,1))}\sum_{b=0}^\infty \sum_{\hhh \in \Xi_b} \mathbb{E}(\overline{\fii}_\sigma^{s_n}(\hhh)^2\overline{\fii}_\sigma^{d}(\hhh)^{-1})\\
	  &\leq\Bigl(\frac{\ualpha}{\lalpha}\Bigr)^{s_n(p+p')}\frac{C^d3^d}{\LL^d(B(0,1))}\sum_{b=0}^\infty\ualpha^{(s_n-d)b} \sum_{\hhh \in \Xi_b} \mathbb{E}(\overline{\fii}_\sigma^{s_n}(\hhh))\\
	  &=\Bigl(\frac{\ualpha}{\lalpha}\Bigr)^{s_n(p+p')}\frac{C^d3^d}{\LL^d(B(0,1))}\sum_{b=0}^\infty\ualpha^{(s_n-d)b}<\infty
  \end{align*}
  as wished.
\end{proof}

\noindent
Combining Propositions~\ref{lem:ub}, \ref{lem:lb}, and \ref{thm:positive-lebesgue} completes the proof of Theorem~\ref{thm:main}.

\bibliographystyle{abbrv}
\bibliography{Bibliography}

\end{document}